\documentclass[a4paper]{article}

\usepackage[english]{babel}
\usepackage[utf8]{inputenc}
\usepackage{amsmath, amsthm, amssymb, bbm, mathtools, tikz-cd}
\usepackage{hyperref}
\usepackage[mathscr]{euscript}
\usepackage{graphicx}
\usepackage{mathrsfs}
\usepackage{url}
\usepackage{upgreek}
\usepackage{thmtools}

\declaretheorem[style=definition,numberwithin=section]{definition}
\declaretheorem[style=definition,qed=$\triangleleft$,sibling=definition]{example}
\declaretheorem[style=definition,qed=$\triangleright$,sibling=definition]{remark}
\declaretheorem[sibling=definition]{theorem}
\declaretheorem[sibling=definition]{lemma}
\declaretheorem[sibling=definition]{proposition}

\declaretheorem[sibling=definition]{conjecture}
\declaretheorem[sibling=definition]{corollary}

\let\sp\relax

\DeclareMathOperator{\spec}{spec}
\DeclareMathOperator{\spf}{spf}
\DeclareMathOperator{\sp}{sp}
\DeclareMathOperator{\spa}{spa}

\DeclareMathOperator{\dimension}{dim}

\DeclareMathOperator{\rank}{rank}

\newcommand{\notp}{{ \{p\} }}

\newcommand{\sset}{\ensuremath{\EuScript{S}}}
\newcommand{\ssettrunc}{\ensuremath{\sset^{<\infty}}}
\newcommand{\ssetfc}{\ensuremath{\sset^{\mathrm{fc}}}}
\newcommand{\ssetpfc}{\ensuremath{\sset^{p-\mathrm{fc}}}}
\newcommand{\ssetnotpfc}{\ensuremath{\sset^{\notp-\mathrm{fc}}}}

\newcommand{\prosset}{\ensuremath{\mathrm{pro}\, \, \sset}}
\newcommand{\prossettrunc}{\ensuremath{\mathrm{pro}^{\natural}\sset}}
\newcommand{\profsset}{\ensuremath{\sset^{\wedge}}}
\newcommand{\profssetp}{\ensuremath{\sset^{\wedge}_{p}}}
\newcommand{\profssetnotp}{\ensuremath{\sset^{\wedge}_{\notp}}}

\newcommand{\ltop}{\ensuremath{\EuScript{LT}\mathrm{op}}}
\newcommand{\rtop}{\ensuremath{\EuScript{RT}\mathrm{op}}}
\newcommand{\procat}{\ensuremath{\mathrm{pro} \,\,}}

\DeclareMathOperator{\site}{\mathbf{site}}
\DeclareMathOperator{\topos}{\uptau}
\DeclareMathOperator{\ootopos}{\uptau_{\infty}}
\DeclareMathOperator{\hootopos}{\uptau_{\infty}^{\wedge}}

\DeclareMathOperator{\Ver}{Ver}
\DeclareMathOperator{\shape}{sh}
\DeclareMathOperator{\protruncshape}{sh^{\natural}}
\DeclareMathOperator{\profiniteshape}{{sh}^{\wedge}}
\newcommand{\pprofiniteshape}{\ensuremath{\profiniteshape_p}}
\newcommand{\notpprofiniteshape}{\ensuremath{\profiniteshape_\notp}}

\DeclareMathOperator{\ettoptype}{\acute{e}t}
\DeclareMathOperator{\profettoptype}{\widehat{\acute{e}t}}

\DeclareMathOperator{\proffkettoptype}{\widehat{ak\acute{e}t}}

\DeclareMathOperator{\qettoptype}{q\acute{e}t}
\DeclareMathOperator{\profqettoptype}{\widehat{q\acute{e}t}}

\DeclareMathOperator*{\colim}{colim}

\DeclareMathOperator*{\hocolim}{hocolim}

\DeclareMathOperator{\kernel}{ker}

\DeclareMathOperator{\map}{Map}
\DeclareMathOperator{\cosk}{cosk}

\title{Etale homotopy theory of non-archimedean analytic spaces}

\author{Joe Berner}

\date{\today}

\begin{document}
\maketitle

\begin{abstract}
We review the shape theory of $\infty$-topoi, and relate it with the usual cohomology of locally constant sheaves. Additionally, a new localization of profinite spaces is defined which allows us to extend the \'etale realization functor of Isaksen. We apply these ideas to define an \'{e}tale homotopy type functor $\ettoptype(\mathcal{X})$ for Berkovich's non-archimedean analytic spaces $\mathcal{X}$ over a complete non-archimedean field $K$ and prove some properties of the construction. We compare the \'etale homotopy types coming from Tate's rigid spaces, Huber's adic spaces, and rigid models when they are all defined.
\end{abstract}

\section{Introduction}

The \'etale homotopy type of an algebraic variety is one of many tools developed in the last century to use topological machinery to attack algebraic questions. For example sheaves and their cohomology, topoi, spectral sequences, Galois categories, algebraic $K$-theory, and Chow groups. The \'etale homotopy type enriches \'etale cohomology and torsor data from a collection of groups into a pro-space. This allows us to almost directly apply results from algebraic topology, with the book-keeping primarily coming from the difference between spaces and pro-spaces. In this paper, we apply this in the context of non-archimedean geometry. The major results are two comparison theorems for the \'etale homotopy type of non-archimedean analytic spaces, and a theorem indentifying the homotopy fiber of a smooth proper map with the \'etale homotopy type of any geometric fiber. The first comparison theorem concerns analytifications of varieties, and states that after profinite completion away from the characteristic of the field that the \'etale homotopy type of a variety agrees with that of its analytification. The second concerns formal models of rigid spaces, and states that after profinite completion away from the characteristic of the residue field that the kummer \'etale homotopy type of a suitable formal scheme agrees with that of its generic fiber.

Because of the syncretic nature of the topic, we have tried to include detailed proofs and references. As many of the documents cited are preprints, we have also reproduced several definitions and theorems. This has the unfortunate effect of engorging this document, but it is hoped to be relatively self-contained.

\subsection{Conventions and notations}

\begin{tabular}{l l}
	$\spec R$	& The spectrum of a commutative ring with unit $R$ \\
    $\spf R$ & The formal spectrum of a topological ring $R$ \\
    $\sp R$ & The analytic spectrum of an affinoid domain $R$ \\
    $\spa R$ & Huber's adic spectrum of an affinoid ring $R=(R^{\triangleright},R^+)$ \\
    $\topos S$ & The topos of sheaves associated to a site $S$ \\
    $\ootopos S$ & The $\infty$-topos associated to a site $S$\\
    $\hootopos S$ & The hypercomplete $\infty$-topos associated to a site $S$\\
    $\procat\EuScript{C}$ & The pro-category of an $\infty$-category $\EuScript{C}$\\
    $\sset$ & The $\infty$-category of $\infty$-groupoids\\
    $\prosset$ & The pro-category of the $\infty$-category of simplicial sets \\
    $\mathbbm{1}_C$ & The terminal object of a $\infty$- or usual category $C$ \\
    $\rtop$ & The $\infty$-category of $\infty$-topoi, morphisms are direct image functors \\
    $\ltop$ & The $\infty$-category of $\infty$-topoi, morphisms are inverse image functors \\
    $\ettoptype X$ & The \'etale homotopy type of a scheme $X$\\
\end{tabular}

\hfill

The end of proofs will be denoted with a tombstone $\square$, the end of remarks with a triangle pointing right $\triangleright$, and the end of examples with a triangle pointing left $\triangleleft$.

Following the style of homotopy theorists, we will always denote colimits via $\colim$, and \emph{never} via the direct limit notation $\lim_{\rightarrow}$. The same style will be used for homotopy colimits, where we write $\hocolim$.

The term $\infty$-category should always be interpreted in the model of $(\infty,1)$-categories given by \emph{quasicategories}, whose theory is carefully worked out in \cite{lurie2009higher}.

We will use the term \emph{quasi-smooth} to mean the same thing as what is called \emph{almost smooth} in \cite{berkovich1993etale}. The hope here is to emphasize its similarity to the notion of `quasi-\'etale' from the later paper \cite{berkovich1996vanishing} and its technical dissimilarity from the more recent notion of `almost $k^\circ$-log smooth'. Ducros has also introduced a class of morphism of the same name in \cite{ducrosbook}. While our class is subsumed by his, we do not know if the converse is true.

As a general visual organizing principle, schemes will be assigned symbols in the usual italicized font, formal schemes will be assigned symbols in Fraktur, and non-archimedean analytic spaces will be given calligraphic symbols. As a rule of thumb, usual 1-categories will be given symbols in the usual italicized format, $\infty$-categories in euscript. We will frequently use the phrase ``$\EuScript{C}$ is modeled by $M$'' to mean that some $\infty$-category $\EuScript{C}$ is the homotopy coherent nerve of the simplicial model category $M$.

\begin{remark}
Let $\EuScript{C}$ be an accessible $\infty$-category that has finite limits, then the pro-category of $\EuScript{C}$ is equivalent to the opposite category of the full subcategory of left exact and accessible functors from $\EuScript{C}$ to $\sset$.

\[ \mathrm{Pro}(\EuScript{C}) \simeq \mathrm{Fun}^{l.acc}(\EuScript{C},\sset)^{op} \]

Since $\infty$-functor categories are only well behaved if one restricts to accessible functors between accessible categories, we will often simply write $\mathrm{Fun}^{lex}(\EuScript{C},\sset)^{op}$ for the $\infty$-category on the right hand side. The `constant diagram' functor $c$ is simply the co-Yoneda functor $cX = \map_\EuScript{C}(X,-)$.

See Proposition 3.1.6 of \cite{xiii2011rational} for details of the proof of this equivalence.
\end{remark}

Let $S$ be a noetherian adic ring and let $T$ be an adic $S$-algebra. We will say that $T$ \emph{is topologically of finite type} over $S$ if for any ideal of definition $\mathfrak{s} \triangleleft S$ that $\mathfrak{s}T$ is an ideal of definition for $T$ and $T/\mathfrak{s}T$ is finitely generated over $R$. In particular if $S$ is a discretely valued ring with maximal ideal $\mathfrak{m} \triangleleft S$ the algebra $T=S[[X]]$ with the ideal of definition $\mathfrak{m}+(x)$ is \emph{not} topologically of finite type over $S$.

\subsection{Summary and Structure}

Section 1 is primarily homotopy theoretic and $\infty$-categorical in nature. We start by introducing the concept of the shape of a site or of a topos and explicitly compare the $\infty$-categorical definition with classical ones. The shape is most naturally a pro-space, however it is natural to consider localizations of the category of pro-spaces. We re-introduce several localizations, and different types of equivalences of shapes corresponding to those localizations. We introduce a new $\infty$-category of $\notp$-profinite spaces, along with completions, modeling completion away from a prime $p$. This gives an easy proof of the following theorem.

\begin{theorem}[Theorem \ref{notpprofinitepi1}]
Let $X$ be a profinite space with $X^{\vee}_{\notp}$ its $\notp$-profinite completion and $x \in X$ a point. Then, the induced map $\pi_1(X,x) \rightarrow \pi_1(X^{\vee}_{\notp},x)$ witnesses the latter pro-group as the maximal prime-to-$p$ completion of the former profinite group.
\end{theorem}

It also allows us to factor Isaksen's \'etale realization functor for motivic spaces in \cite{isaksen2004etale}.

\begin{theorem}[Theorem \ref{etalerealization}]
Let $S$ be a separated noetherian scheme over $\mathbb{F}_p$, and write $\EuScript{S}\mathrm{pc}_{S,*}^{\mathbb{A}^1}$ for the $\infty$-category of pointed motivic spaces over $k$, then the $\notp$-profinite \'etale homotopy type gives an \'etale realization functor
\[ \profettoptype_\notp : \EuScript{S}\mathrm{pc}_{S,*}^{\mathbb{A}^1} \to \profssetnotp \]
which lifts Isaksen's $\ell$-adic \'etale realization functor to $\notp$-profinite spaces.
\end{theorem}

Since  geometers typically do not use $\infty$-categories, so we set up a dictionary which allows us to deduce $\infty$-categorical equivalences using $1$-categorical torsors and abelian sheaf cohomology. The $\infty$-categorical point of view makes it completely formal that the \'etale homotopy type functor is a `sheaf of pro-spaces' on the \'etale site, which we extend to prove a hypercover descent statement generalizing one of Isaksen.

Section 2 is concerned with studying the \'etale homotopy type of a general non-archimedean $K$-analytic space $\mathcal{X}$ for $K$ a complete non-trivially valued non-archimedean field. There are actually two rival classes of morphisms, \emph{quasi-\'etale} and \emph{\'etale} which give different topoi but are known to give equivalent abelian sheaf cohomology. For a general $\mathcal{X}$ we simply do not know if the protruncated \'etale homotopy type is equivalent to the protruncated quasi-\'etale homotopy type. In the case where $\mathcal{X}$ is hausdorff and strictly $K$-analytic, we prove the following result that is of independent interest and implies the two rival homotopy types are equivalent in the case we need.

\begin{theorem}[Theorem \ref{comparison-berkovich-adic}]
Let $\mathcal{X}$ be a hausdorff strictly $K$-analytic space over a complete non-archimedean field $K$, and write $\mathcal{X}^{ad}$ for the corresponding adic space.

Then the quasi-\'etale 1-topos of $\mathcal{X}$ is equivalent to the \'etale 1-topos of $\mathcal{X}^{ad}$.
\end{theorem}

As this suggests we also compare the theories obtained from the rigid and adic formalisms of non-archimedean geometry to Berkovich's. We further prove an Artin Comparison theorem, compare to de Jong's \'etale fundamental group, and end by stating a conjecture regarding an analogue of Friedlander's homotopy fiber theorem.

Section 3 is concerned with a complete discretely valued $K$ and using formal models. We rapidly retread the theory of formal kummer \'etale morphisms of special formal schemes as set up in \cite{berkovichcomplex}. We use many of the same geometric ideas of \cite{berkovichcomplex} to prove an abelian cohomological comparison, but with different techniques. We must use fundamentally new ideas to compare fundamental groups by passing through the usual scheme theoretic spectrum of the complete rings that locally model the formal scheme. Here we can use log geometry to descend a finite \'etale cover on $\mathcal{X}$ to a formal kummer \'etale one on the formal model. We deduce the following theorem.

\begin{theorem}[Theorem \ref{globalcomparisonberklog}]
Let $\mathfrak{X}$ be a algebraizably formally $R$-log smooth special formal scheme with fine log structure. Consider its generic fiber $\mathcal{X} = \mathfrak{X}_\eta$ as a non-archimedean $K$-analytic space.

The $\notp$-profinitely completed quasi-\'etale homotopy type of $\mathcal{X}$ agrees with the $\notp$-profinitely completed adically kummer \'etale homotopy type of $\mathfrak{X}$, where $p$ is the characteristic of the residue field $k$.

\[ \profettoptype_{\notp} \mathcal{X} \simeq \proffkettoptype_{\notp} \mathfrak{X} \]
\end{theorem}
\subsection{Acknowledgements}

I am grateful to Vladimir Berkovich, Elden Elmanto, Marc Hoyois, and Yoichi Mieda for patience with naive questions and for thoughtful answers to them. I thank the organizers and speakers of the Arizona Winter School of 2017 for hosting a wonderful conference which led to several useful conversations. I must also thank David Hansen who left a very helpful answer on MathOverflow, which greatly simplified a proof. Most of all I thank my advisor Henri Gillet for patience with foolish questions and thoughtful answers to them.

\section{Shapes assigned to topoi}

We will briefly review and define some of the relevant objects and ideas we will use later on.

\subsection{Review of literature}

The idea of the shape of an $\infty$-topos goes back to \cite{toen2002segal}, where it was defined in model categorical language. Later in \cite{lurie2009higher}, Lurie defines for an $\infty$-topos $\EuScript{T}$ with its unique geometric morphism $\pi : \EuScript{T} \to \sset$ the shape of the $\infty$-topos to be the \emph{functor}
\[ \pi_* \pi^* : \sset \to \sset \]
Lurie's definition is absolute in the sense that the shape is defined for an $\infty$-topos, but not explicitly for an arbitrary geometric morphism. In \cite{hoyois2015higher}, Hoyois gives a definition of what may be called the relative shape for a morphism of $\infty$-topoi. We explain this perspective in the next paragraph

The $\infty$-topos $\EuScript{T}$ admits a contractible space of geometric morphisms to the terminal $\infty$-topos of $\infty$-groupoids $\EuScript{S}$. Call any such geometric morphism $\Pi$, with the adjoint pair $\Pi^* \dashv \Pi_*$. The left-adjoint `constant sheaf' functor $\Pi^*$ preserves colimits by virtue of being a left adjoint, and in addition is left-exact. This left-exactness implies that it further admits a pro-left adjoint denoted $\Pi_! : \EuScript{T} \rightarrow \prosset$. It is difficult to give an explicit description of this functor, but we will characterize some of its values using $1$-categorical language. The key fact is that $\Pi_! \mathbbm{1}_\EuScript{T}$ co-represents the functor $\pi_* \pi^*$ in the sense that
\[ \map_{\prosset}(\Pi_! \mathbbm{1}_\EuScript{T}, c(-)) \simeq  \pi_* \pi^* \]
where $c$ is the `constant diagram' functor $c : \sset \to \prosset$. We write
\[ \Pi_\infty : \rtop  \to \prosset \]
%for the functor sending an $\infty$-topos $\EuScript{T}$ to $\Pi_! \mathbbm{1}_\EuScript{T}$. We deduce that the functor $\Pi_\infty : \rtop \to \prosset$ is the left adjoint to $\iota : \prosset \to \rtop$, where $\iota$ is the extension to $\prosset$ of the functor $\iota' : \sset \to \rtop$ which sends a space to the slice $\infty$-topos
%
\[ \iota' : X \mapsto \sset/X \]
We will repeatedly use that $\Pi_\infty$ and $\Pi_!$ preserve colimits, which is immediately implied by each being a left adjoint.

We will use the two to define shapes assigned to three types of objects.

\begin{definition}
Using the above notations and context:
\begin{enumerate}
\item Given an $\infty$-topos $\EuScript{T}$, define $\shape(\EuScript{T})$ the \emph{shape of} $\EuScript{T}$, to be the pro-space $\Pi_\infty (\EuScript{T})$.

\item Given a site $S$, define the \emph{shape assigned to the site} $S$ to be the shape assigned to the hypercomplete $\infty$-topos of sheaves of $\infty$-groupoids on $S$, that is ${Sh}^{\wedge}_\infty (S)$.

\item Given a 1-topos $T$, define the \emph{shape assigned to} $T$ to be the shape assigned to the hypercomplete $\infty$-topos of sheaves of $\infty$-groupoids on $T$. In model categorical language the hypercomplete $\infty$-topos can be presented using the category $sT$ of simplicial sheaves with the local injective structure. The $\infty$-categorical description is as the hypercompletion of the $\infty$-topos of limit preserving functors $\EuScript{F} \in \mathrm{Fun}^{R}(N(T^{op}),\sset )$.
\end{enumerate}

There is a potential gap in the case that we have a site $S$ presenting a topos $T$. To see that their shapes are equivalent, levelwise sheafification gives a Quillen equivalence between simplicial sheaves with local injective model structure and simplicial pre-sheaves with hyperlocal injective model structure. The equivalence of homotopy categories is stated as Corollary 1.14 of \cite{jardine1987simplical}, and a proof can be written following in its steps. Given the Quillen equivalence, Theorem 6.5.2.1 and Corollary A.3.1.12 of \cite{lurie2009higher} say that the simplicial nerve of either simplicial model category is equivalent to hypercomplete $\infty$-topos of sheaves on the site.
\end{definition}

We make the following definition, in the spirit of the original definition of the \'etale homotopy type via the Verdier functor of \cite{artin1969etale}.

\begin{definition}
If $S$ is a locally connected site with fibered products and connected components functor $\pi_0 : S \rightarrow \mathrm{Set}$ then we can define the \emph{naive shape of} $S$ as follows. This is the diagram of simplicial sets, defined as the Verdier functor
\[ \Ver_S : HR(S) \rightarrow sSet \]
which assigns to a hyper-cover $U_\bullet \in HR(S)$ the simplicial set of `level-wise connected components'
\[ U_\bullet \mapsto \Big[ [n] \mapsto \pi_0(U_n)\Big] \]
\end{definition}
To clarify its relation to the previous $\infty$-categorical definitions, we appeal to the following theorem of \cite{hoyois2015higher}:

\begin{theorem}\label{verdiercorepresents}
The diagram $\Ver_S$ naturally co-represents the shape assigned to $S$.
\end{theorem}

This follows directly from the `generalized Verdier hypercover theorem', variants of which go back to work of Brown for spectral presheaves. We need an unstable analogue, see for example Theorem 8.6 (b) of \cite{dhi2004}. Given this theorem, we see that Verdier's Theorem 7.4.1 of \cite{artin1973theorie} immediately suggests the following proposition.

\begin{proposition}
Let $S$ be a site with products, and assume that $S$ is locally connected. Write $LS(\shape S)$ for the class of local systems of sets on $\shape S$, and write $LC(S)$ for the class of locally constant sheaves of sets on $S$. Give both the structure of a category by adding morphisms which for both categories are only the morphisms which are locally constant

The shape $\shape S$ has the property that there is an assignment $\Phi : LC(S) \rightarrow LS(\shape S)$, and furthermore

\begin{enumerate}
\item the functor $\Phi$ is an equivalence;

\item the functor $\Phi$ preserves cohomology in the sense that for any locally constant abelian sheaves $\mathscr{L}$ and $\mathscr{L}'$ and a locally constant morphism $\phi : \mathscr{L} \to \mathscr{L}'$ we have the following commutative diagram

\begin{center}
\begin{tikzcd}
H^i_{sheaf}(S,\mathscr{L}) \arrow[r, "\sim"] \arrow[d, "\phi"] & H^i_{sing}(\shape S, \Phi \mathscr{L}) \arrow[d, "\Phi\phi"] \\
H^i_{sheaf}(S,\mathscr{L}') \arrow[r, "\sim"] & H^i_{sing}(\shape S, \Phi \mathscr{L}')
\end{tikzcd}
\end{center}

where the horizontal maps are isomorphisms.
\end{enumerate}
\end{proposition}

This proposition is really just the special case of $n=0$ in Proposition 2.10 of \cite{hoyois2015higher}, albeit stated and proved in $1$-categorical language.

\begin{proof}
Let $\mathscr{F}$ be a locally constant sheaf of sets on $S$. Choose a covering $\{ V_i \rightarrow \mathbbm{1}_S \}$ on which $\mathscr{F}_i = \mathscr{F}|_{S/V_i}$ is constant with value set $F$, and such that each $V_i$ is connected. We refine this inductively to a hypercovering $U_\bullet$ such that $U_n = \coprod_{i \in I_n} W_i^n$ with each $W_i^n$ connected. We write $\vec{i} = \{i_1,\ldots,i_n\}$ as shorthand, and $U_{\vec{i}}$ for the products $U_{i_1} \times \ldots \times U_{i_n}$. Furthermore, fix isomorphisms $\sigma_{\vec{i}}^f : U_{\vec{i}} \overset{\sim}{\rightarrow} U_{\sigma \vec{i}}$ for any permutation $\sigma$ in the symmetric group. Note that the isomorphism $\sigma_{\vec{i}}^f$ does not depend on the particular $W_i^n$, only on the larger $U_{\vec{i}}$ that it is contained in. We can now define a local system on the simplicial set

\[ \Ver_S\Big( U_\bullet \Big) = \pi_0(U_\bullet)\]

\noindent as follows. The zero simplices of the above simplicial set are the connected components of the $U_i$. On which we can simply define $\Phi \mathscr{F}$ to take the zero simplex $U_i$ to the set $\mathscr{F}(U_i)$, and let the isomorphism on any $n$-simplex $W_i^n$ to be the isomorphism $\sigma_{\vec{i}}^f$ when $W_i^n$ is a connected component in $U_{\vec{i}}$.

We wish to call this $\Phi \mathscr{F}$. We must show that this operation only depends on the sheaf $\mathscr{F}$ and not on the covering splitting it, at least up to equivalence of local systems. However if $\mathscr{F}$ admits another splitting $\{V_i' \rightarrow \mathbbm{1}_S\}$ we can find a common refinement $\{V_i'' \rightarrow \mathbbm{1}_S\}$ of the two, and one checks that the pullback of the local systems from $\Ver_S\Big( U_\bullet \Big)$ and $\Ver_S\Big( U_\bullet' \Big)$ to $\Ver_S\Big( U_\bullet'' \Big)$ agree.

For essential surjectivity, a local system on a pro-space comes from a local system at some level. This means we can lift the data of the local system to descent data for a locally constant sheaf and glue it together

For injectivity, we see that the local system corresponding to $\mathscr{F}$ determines descent data equivalent to $\mathscr{F}$.

This implies the first claim. The second is made of two parts. The easy part is that the construction of $\Phi$ immediately implies the commutativity of the diagram. The slightly harder part is that the horizontal arrows are isomorphisms. This is just the Verdier hypercover theorem, since the singular cohomology of the system $\Phi \mathscr{L}$ on the pro-space is defined to be the colimit of the singular cohomology at each level.
\end{proof}

We have the following corollary.

\begin{corollary}\label{locallyconstantsheaves}
Let $f_* : T \rightarrow T'$ be a geometric morphism of locally connected $1$-topoi, if the inverse image functor $f^*$ induces an equivalence of categories between subcategories of locally constant sheaves, then $f$ induces an isomorphism on fundamental pro-groupoids of the associated $\infty$-topoi.
\end{corollary}

\begin{proof}
The fundamental pro-groupoid of the shape classifies local systems on the shape, which is the same as classifying locally constant sheaves in the topos. Since the latter category is the same, the former must be as well.
\end{proof}

The obvious temptation is to take the naive shape as the definition of the shape and apply this in the context of non-archimedean geometry. This would further have the benefit of avoiding the $\infty$-categorical language. While fibered products exist in the category of hypercovers, one does not have weak equalizers in general as observed in SGA. This difficulty can be circumvented in different ways, the most classical being only considering hypercovers up to simplicial homotopy as originally done in \cite{artin1969etale}. With enough points, one can instead resolve this by using Lubkin covers, which is used in \cite{friedlander1982etale} to define an \'etale topological type. This gives an honest pro-simplicial set, instead of a pro-object of homotopy types or a pro-$\infty$-groupoid. However, the Lubkin covers technique has problems with functoriality with respect to morphisms of sites, and to obtain functoriality one must pass through the homotopy type. The $\infty$-categorical perspective is quite simple to work with structurally. As nothing comes for free in mathematics, the disadvantage is the relative abstraction and the model categorical work required to relate it back to $1$-categorical properties.

We will first set up some notation for the $\infty$-categories of completed spaces we will work with. Except for the formalism of $\notp$-finite spaces, it can be found in more detail in \cite{xiii2011rational}.

\begin{definition}
Let $\sset$ be the $\infty$-category of $\infty$-groupoids, and let $p$ be a prime number.

\begin{enumerate}
\item Denote by $\ssettrunc$ the full subcategory of $\sset$ spanned by truncated spaces, that is spaces $X$ where $\pi_i(X) = 0$ for all sufficiently large $i$. This will be called the $\infty$-category of \emph{truncated spaces}.

\item Denote by $\ssetfc$ the full subcategory of $\ssettrunc$ spanned by spaces $X$ with finitely many connected components and where for each $x \in X$ the group $\prod_{i \geq 1} \pi_i(X,x)$ is finite. This will be called the $\infty$-category of \emph{$\pi$-finite spaces}.

\item Denote by $\ssetpfc$ be the full subcategory of $\ssetfc$ spanned by spaces $X$ where for each $x \in X$ each group $\pi_i(X,x)$ is a $p$-group. This will be called the $\infty$-category of \emph{$p$-finite spaces}.

\item Denote by $\ssetnotpfc$ the full subcategory of $\ssetfc$ generated under finite limits by spaces $X$ which satisfies at least one of the following:

\begin{enumerate}
\item the space $X$ is an $\ell$-finite space for some prime $\ell$ distinct from $p$,

\item the space $X$ is $BG$ for some finite group $G$ whose order is not divisible by $p$.
\end{enumerate}
\noindent This will be called the $\infty$-category of \emph{$\notp$-finite spaces}.
\end{enumerate}
\end{definition}

Let us briefly state some generalities about pro-$\infty$-categories. These results mimic foundational statements about presentable or accessible $\infty$-categories, however pro-categories are essentially never accessible and so these results must be justified.

\begin{lemma}\label{pro-categorylemma}
Let $i: \EuScript{C}' \subset \EuScript{C}$ be a fully faithful inclusion of accessible $\infty$-categories, and further assume that the inclusion of $\EuScript{C}'$ into $\EuScript{C}$ preserves finite limits. Then

\begin{enumerate}
\item the inclusion extends to a fully faithful inclusion $i: \procat \EuScript{C}' \subset \procat \EuScript{C}$,

\item the inclusion in (1) admits a left adjoint $L$ given by precomposition with $i$,

\item the functor $L$ is a localization,

\item a morphism $X_\bullet \to Y_\bullet$ in $\procat \EuScript{C}$ is taken to an equivalence by $L$ if and only if $\colim_{i \in I} \map(Y_i,Z) \to \colim_{j \in J} \map(X_j,Z)$ is an equivalence of spaces for all objects $Z$ in $\EuScript{C}'$,

\item assume that $\EuScript{C}'$ is generated under finite limits by some class of objects $\EuScript{O}$, then $L$ turns a morphism $X_\bullet \to Y_\bullet$ into an equivalence if and only if $\colim_{i \in I} \map(Y_i,Z) \to \colim_{j \in J} \map(X_j,Z)$ is an equivalence of spaces for all objects $Z$ in $\EuScript{O}$.
\end{enumerate}
\end{lemma}

\begin{proof}
Compare to Remark 3.1.7 of \cite{xiii2011rational}, although the adjunction is stated on the wrong side.
\begin{enumerate}
\item We note that as objects in the $\infty$-category $\sset$ we have the following equivalence by Proposition 5.1.1 and 5.1.2 of \cite{barnea2015pro}

\[ \map_{\procat \EuScript{C}'}(X_\bullet,Y_\bullet) \simeq \lim_{j \in J} \colim_{i \in I} \map_{\EuScript{C}'}(X_j,Y_i) \] 

and similarly

\[ \lim_{j \in J} \colim_{i \in I} \map_{\EuScript{C}}(X_j,Y_i) \simeq \map_{\procat \EuScript{C}}(X_\bullet,Y_\bullet)\]

and so $i$ being fully faithful says that the RHS of the first equivalence is equivalent to the LHS of the second equivalence at each $(i,j)$.

\item We want to check that the following spaces are equivalent

\[ \map_{\procat \EuScript{C}} (LX_\bullet , Y_\bullet) \simeq \map_{\procat \EuScript{C}} ( X_\bullet, iY_\bullet) \]

we use the fact that $i$ extends to a limit preserving functor and that mapping spaces preserve limits to reduce to the case that $Y_\bullet$ is co-representable, which is true as a corollary of the $\infty$-categorical Yoneda lemma.

\item A functor is defined to be a localization when it is a left adjoint with fully faithful right adjoint, which we proved in (1) and (2).

\item This is the Yoneda lemma, as the diagram $X_\bullet$ is realized as a left exact accessible functor by $\lim_{j \in J} \map(X_j, -)$.

\item Since the functors which co-represent pro-objects are left exact, they preserve finite limits. In particular for any finite diagram of objects in $\EuScript{O}$, the assumption that $f$ induces equivalences on mapping to objects of $\EuScript{O}$ implies that it induces an equivalence to the limit over the diagram. Repeating the process implies (4). \qedhere
\end{enumerate}
\end{proof}

The model theory of pro-categories can be difficult, as they may be neither fibrantly nor cofibrantly generated. In general the above $\infty$-categorical perspective makes it easy to define various localizations of pro-categories.

\begin{definition}
We now take the associated pro-categories of the above list resulting in the following.

\begin{enumerate}
\item The category of \emph{protruncated spaces}

\[ \prossettrunc = \procat \ssettrunc \simeq \mathrm{Fun}^{lex}(\ssettrunc, \sset)^{op}\]

\item The category of \emph{profinite spaces}

\[ \profsset = \procat \ssetfc \simeq \mathrm{Fun}^{lex}(\ssetfc,\sset)^{op}\]

\item The category of \emph{$\notp$-profinite spaces}

\[ \profssetnotp = \procat \ssetnotpfc \simeq \mathrm{Fun}^{lex}(\ssetnotpfc,\sset)^{op}\]

\item The category of \emph{$p$-profinite spaces}

\[  \profssetp = \procat \ssetpfc \simeq \mathrm{Fun}^{lex}(\ssetpfc,\sset)^{op} \]
\end{enumerate}
\end{definition}

Given a space $X \in \sset$, we complete it by considering its co-Yoneda functor $\map_\sset (X,-) : \sset \rightarrow \sset$ and restricting the domain to truncated spaces $\sset^{<\infty}$, $\pi$-finite spaces $\ssetfc$, $\notp$-finite spaces $\ssetnotpfc$, or $p$-finite spaces $\ssetpfc$. The resulting protruncated space will be denoted $X^\natural$, the resulting profinite space by $X^{\vee}$, the resulting $\notp$-profinite space by $X^{\vee}_{\notp}$, and the resulting $p$-profinite space by $X^{\vee}_p$, respectively. All four completion functors admit a `materialization' right adjoint, formally defined as
\[ \operatorname{Mat}(X) = \map_{\EuScript{C}}(*,X) \]
where $\EuScript{C}$ is protruncated spaces $\prossettrunc$, profinite spaces $\profsset$, $\notp$-profinite spaces $\profssetnotp$, or $p$-profinite spaces $\profssetp$. These can be informally described by taking the limit in $\sset$ over the diagram.

Lemma \ref{pro-categorylemma} implies we have a chain of fully faithful right adjoints
\[ \profsset_\ell \subset \profssetnotp \subset \profsset \subset \prossettrunc \subset \prosset \]
with left adjoint localizations for any pair of distinct prime numbers $p$ and $\ell$,
\[ \prosset  \to \prossettrunc \to \profssetnotp \to \profsset_\ell .\]
Some remarks are in order about these categories.

The $\infty$-category of pro-spaces is modeled by Isaksen's strict model structure, see Theorem 5.2.1 of \cite{barnea2015pro} for a generalization to a class of pro-categories. This is the most natural target for the shape functor, however it is in general difficult to actually compute derived global sections of non-truncated sheaves of $\infty$-groupoids.

The $\infty$-category of profinite spaces can be modeled via the model category described in \cite{quick2008profinite} and \cite{quick2011continuous}. In this, the underlying category is the category of simplicial objects in the category of profinite sets. That Quick's model category models the $\infty$-categorical description of profinite spaces follows from Theorem 7.4.7 and Corollary 7.2.12 of \cite{barnea2015pro}.

Similarly, in \cite{morel1996ensembles} Morel defined a model structure on simplicial profinite sets to model $p$-adic completion. This also models the $\infty$-category of $p$-profinite spaces as described above. This follows from Theorem 7.4.10 and Corollary 7.3.8 of \cite{barnea2015pro}.

As for protruncated spaces, it is in general difficult to detect whether or not a map of pro-spaces is an equivalence. Recall that in classical topology one of the main theorems of \cite{may-whitehead} states extends the Whitehead Theorem,

\begin{theorem}
Let $f: X \rightarrow Y$ be a map of fibrant simplicial sets. Then $f$ is a homotopy equivalence if and only if it induces an isomorphism on fundamental groupoids and for any local system of abelian groups $A$ on $Y$, the map $f$ induces isomorphisms on cohomology groups for every $n \geq 0$

\[ H^i(Y,A) \cong H^i(X,f^* A) \]
\end{theorem}

This is false in pro-spaces, where the map from a space $X$ to its Postnikov tower $P_\bullet X$ is a Whitehead equivalence but will usually not admit a map back from the tower. We can force the above theorem to be true by localization, which we will now prove to be equivalent to protruncation. This is mentioned as a remark in a publicly available draft of Marc Hoyois, but the author does not know of a published proof. For completeness' sake we include the easy proof.

\begin{theorem}\label{nonstrictmodels}
The non-strict model structure on pro-simplicial sets of \cite{isaksen2003strict} models the $\infty$-category $\hat{\sset}^{<\infty}$.
\end{theorem}

\begin{proof}
Isaksen's non-strict model structure is a left Bousfield localization of the strict model structure at truncated equivalences. Similarly, the restriction functor from $\prosset$ to to $\prosset^{<\infty}$ is a left exact $\infty$-categorical localization, and it is the localization exactly at truncated equivalences. Proposition 5.2.4.6 of \cite{lurie2009higher} tells us the left Bousfield localization maps to an $\infty$-categorical localization. Since localizations are universal with respect to the class they invert, the two must be equivalent.
\end{proof}

As far as $\notp$-profinite spaces, the original definition was modified after the author saw \cite{carchedi-elmanto}. In particular, the author expects that $\notp$-profinite spaces should be a further localization of their category. We justify this definition by the following series of results, after we give a mapping space definition of the fundamental group of a $\notp$-profinite space in the expected way.

\begin{definition}
The \emph{fundamental group} of a pointed $\notp$-profinite space $(X,x)$

\[ \pi_1(X,x) := \pi_0 \map([S^1]^{\vee}_{\notp},X)\]
\end{definition}

We can now compute the fundamental group after $\notp$-profinite completion.

\begin{theorem}\label{notpprofinitepi1}
Let $X$ be a profinite space with $X^{\vee}_{\notp}$ its $\notp$-profinite completion and $x \in X$ a point. Then,

\begin{enumerate}
\item the pro-group $\pi_1(X^{\vee}_{\notp},x)$ is profinite,

\item the induced map $\pi_1(X,x) \rightarrow \pi_1(X^{\vee}_{\notp},x)$ witnesses the latter pro-group as the maximal prime-to-$p$ completion of the former profinite group.
\end{enumerate}
\end{theorem}

\begin{proof}
We immediately reduce to the connected case. If $X$ is a profinite space, we may represent it with some $1$-categorical diagram of $\pi$-finite spaces $X_\bullet:I \to \ssetfc$. Write $N_p \pi_1 X_i$ for the smallest normal subgroup of $\pi_1 X_i$ containing any and thus all of the Sylow $p$-subgroups. For each $i \in I$, we have a natural map $X_i \to B \pi_1 X_i$. Now, $BN_p \pi_1 X_i$ also maps into $B\pi_1 X_i$, by applying $B$ to the inclusion $N_p \pi_1 X_i \to \pi_1 X_i$. Finally, write $N_p X_i$ for the homotopy pull back. We have a natural map
\[ N_p X_i \to X_i \]
Applying the Mayer-Vietoris sequence for the fundamental group quickly shows that $\pi_1 N_p X_i \cong N_p \pi_1 X_i$ and the map above realizes this isomorphism. See Proposition 2.2.2 of \cite{may2011more} for details of the referenced sequence.

Write $CX_i$ for the homotopy cofiber of $N_p X_i \to X_i$. It is straightforward to check that we get a cofiber sequence of pro-simplicial sets in Quick's model structure. One checks $\pi_1$ in the obvious way, and one uses the $5$-lemma on the two long exact sequences for cohomology, noting that filtered colimits preserve long exact sequences. So we have a homotopy cofiber sequence in profinite spaces
\[ N_p X_\bullet \to X_\bullet \to CX_\bullet  \]
Since $\pi_1$ commutes with homotopy cofiber sequences, it is now enough to check that each $N_p X_i$ becomes simply connected after $\notp$-completion. But this is clear, as $\pi_1 N_pX_i$ is a quasi-$p$-group and so admits no maps to groups whose order is not divisible by $p$. In particular, we see that if we take a particular representation of $\pi_1(X,x)$ as a cofiltered system of finite groups that the above procedure exactly kills the $p$-parts of all groups in the diagram.
\end{proof}

\begin{example}
The map $BA_5 \to BA_5^+$ is a $\{3\}$-equivalence, where $-^+$ is Quillen's plus construction. In particular, one can apply the universal coefficient theorem and the Hurewicz theorem to conclude $\pi_2([BA_5]^\vee_{\{3\}}) \simeq \mathbb{Z}/2$.
\end{example}

\begin{theorem}\label{etalerealization}
Let $S$ be a separated noetherian finite dimensional scheme over $\mathbb{F}_p$, and write $\EuScript{S}\mathrm{pc}_{S,*}^{\mathbb{A}^1}$ for the $\infty$-category of pointed motivic spaces over $S$, then the $\notp$-profinite \'etale homotopy type gives an \'etale realization functor
\[ \profettoptype_\notp : \EuScript{S}\mathrm{pc}_{S,*}^{\mathbb{A}^1} \to \profssetnotp \]
which lifts Isaksen's $\ell$-adic \'etale realization functor to $\notp$-profinite spaces.
\end{theorem}

\begin{proof}
Marco Robalo's thesis \cite{robalo-thesis} identifies the $\infty$-category of motivic spaces as (non-hypercompleted) Nisnevich sheaves of spaces on $Sm/S$ localized at the class of all projections $X \times \mathbb{A}^1 \to X$. By Corollary 2.25 of \cite{dag-xi}, the Nisnevich $\infty$-topos is already hypercomplete. Since the \'etale homotopy type satisfies hypercover descent for the \'etale topology, it \emph{a fortiori} satisfies descent for the Nisnevich topology. Thus we just need to show $\mathbb{A}^1$ homotopy invariance of the $\notp$-profinite \'etale homotopy type.

First, we observe that \cite{isaksen2004etale} proves the $\ell$-adic case. By his arguments for the $\ell$-adic case, this reduces the problem to computing the $G$-torsors of a smooth and geometrically pointed $S$-scheme $(X,x)$ and $(X \times \mathbb{A}^1, (x,0))$. In general it is very difficult to understand even the prime-to-$p$ quotients of products, however in this case we may apply Proposition XIII.4.1 of \cite{raynaud1971revetements} to deduce a certain exact sequence of profinite groups
\[ \pi_1^{p'} (\mathbb{A}^1_{\overline{k}},0) \to \pi_1' (X \times \mathbb{A}^1,(x,0)) \to  \pi_1(X,x)  \to 1\]
Now $\mathbb{A}^1$ over an algebraically closed field admits no covers of order prime to $p$, so we get an isomorphism between $\pi_1'(X \times \mathbb{A}^1,(x,0))$ and the profinite \'etale fundamental group $\pi_1(X,x)$. From this part of the proof we suppress the choice of base-points for notational convenience. The profinite group $\pi_1'(X \times \mathbb{A}^1)$ is defined to be $\pi_1(X \times \mathbb{A}^1)/N$ where $N$ is the subgroup
\[ N = \kernel \Big( \kernel(\pi_1(X \times \mathbb{A}^1) \to \pi_1(X)) \to \kernel (\pi_1(X \times \mathbb{A}^1) \to \pi_1(X))^{p'}\Big) \]
At this point we have the following exact sequence.
\[ 1 \to N \to \pi_1(X \times \mathbb{A}^1) \to \pi_1(X) \to 1 \]
Applying prime-to-$p$ quotients leaves it only right exact.
\[ N^{p'} \to \pi_1^{p'}(X \times \mathbb{A}^1) \to \pi_1^{p'}(X) \to 1 \]

However $N$ was the objects of the kernel of $\pi_1(X \times \mathbb{A}^1) \to \pi_1(X)$ which vanish after taking prime-to-$p$ quotients, and so the left map becomes zero after taking prime-to-$p$ quotients. We thus have an isomorphism
\[ \pi_1^{p'}(X \times \mathbb{A}^1) \simeq \pi_1^{p'}(X) \]
\end{proof}

Note that this will be further lifted by \cite{carchedi-elmanto}.

As another technical justification for the naturality of this construction, Theorem 2.4 of \cite{quickstableetale} allows us to perform a left Bousfield localization at any set of morphisms. This uses that Quick's model category is left proper, simplicial, and fibrantly generated which are proved in \cite{quick2008profinite} and \cite{quick2011continuous}. In fact it follows as a corollary of $\profsset$ being modeled by Quick's model category that $\profssetnotp$ is modeled by the left Bousfield localization at the given set of generators of $\ssetnotpfc$.

\subsection{Preliminary results on shapes}

We begin with the basic definitions of the objects and categories we will work with.

\begin{definition}
The following defines and names various localizations of the shape.
\begin{itemize}
\item We define the \emph{protruncated shape} $\protruncshape \EuScript{T}$ of an $\infty$-topos (resp. assigned to a 1-topos, resp. assigned to a site) to be the protruncated completion of the shape of the $\infty$-topos (resp. assigned to the 1-topos, resp. assigned to the site).

\item Define the \emph{profinite shape} $\profiniteshape \EuScript{T}$ of an $\infty$-topos (resp. assigned to a 1-topos, resp. assigned to a site) to be the profinite completion of the shape of the $\infty$-topos (resp. assigned to the 1-topos, resp. assigned to the site).

\item Define the \emph{$\notp$-profinite shape} $\notpprofiniteshape \EuScript{T}$ of an $\infty$-topos (resp. assigned to a 1-topos, resp. assigned to a site) to be the $\notp$-profinite completion of the shape of the $\infty$-topos (resp. assigned to the 1-topos, resp. assigned to the site).

\item Similarly, define the \emph{$p$-profinite shape} $\pprofiniteshape \EuScript{T}$ of an $\infty$-topos (resp. assigned to a 1-topos, resp. assigned to a site) to be the $p$-profinite completion of the shape of the $\infty$-topos (resp. assigned to the 1-topos, resp. assigned to the site).

\item A morphism of 1-topoi $f : T \rightarrow T'$ is called a \emph{shape equivalence} (resp. \emph{protruncated shape equivalence}, resp. \emph{profinite shape equivalence}) if the induced map on the shapes (resp. protruncated shapes, resp. profinite shapes) assigned to the hypercomplete $\infty$-topoi is a weak equivalence.
\end{itemize}
\end{definition}

The following propositions give us a way to verify the above criterion in the classical language of torsors and abelian sheaf cohomology. They follow mostly formally from the literature, but we state them for completeness' sake.

\begin{proposition}\label{shapeequivalence}
A morphism of locally connected topoi $f : T \rightarrow T'$ is a protruncated (resp. profinite) shape equivalence if and only if the following three conditions hold,

\begin{enumerate}
\item For any locally constant sheaf $\mathscr{S} \in T'$ of sets (resp. finite sets), we have $H^0(T',\mathscr{S}) \overset{\sim}{\rightarrow} H^0(T,f^* \mathscr{S})$ via the natural map,

\item for any group $G$ (resp. finite group) with associated constant sheaf $\mathscr{G} \in T'$, we have $\check{H}^1(T',\mathscr{G}) \overset{\sim}{\rightarrow} \check{H}^1(T,f^*\mathscr{G})$ via the natural map,

\item and for any locally constant sheaf $\mathscr{L} \in T'$ of abelian groups (resp. finite abelian groups of prime power order), we have that 
\[ H^q(T',\mathscr{L}) \overset{\sim}{\rightarrow} H^q(T,\mathscr{L})\]
for all $q \geq 0$ via the natural map.
\end{enumerate}
\end{proposition}

\begin{proof}[Proof of protruncated case]
This is almost the exact criteria for a Whitehead equivalence in Isaksen's non-strict model category. Isaksen requires the map induce an isomorphism on fundamental pro-groupoids. But fundamental pro-groupoids are determined functorially by their maps to $BG$ spaces as $G$ varies over all groups.
\end{proof}

\begin{proof}[Proof of profinite case]
These are almost exactly the criteria for a map to be a weak equivalence in Quick's profinite spaces category. The difference is the same as in the protruncated case, once we show that we only need to check cohomology for sheaves of finite abelian $\ell$-primary groups. Quick's criterion would require an arbitrary locally constant sheaf $\mathscr{K}$ of finite abelian groups, but $\mathscr{K}$ is the direct sum of its $p$-primary submodule sheaves as $p$-ranges over all primes.
\end{proof}

\begin{proposition}\label{pprofshapeequivalence}
Let $f : T \rightarrow T'$ be a morphism of locally connected topoi. The following are equivalent,
\begin{enumerate}
\item the geometric morphism $f$ is a $p$-profinite shape equivalence,

\item the geometric morphism $f$ induces an isomorphism on $\mathbb{Z}/p$ cohomology in all degrees,

\[ f^* : H^i(T',\mathbb{Z}/p) \overset{\sim}{\rightarrow} H^i(T, \mathbb{Z}/p) \]
\end{enumerate}
\end{proposition}

\begin{proof}
Unlike the above, this is not a direct tautology. The weak equivalences in Morel's model category are morphisms inducing isomorphisms on $H^i(-,\mathbb{Z}/p)$ cohomology groups. The equivalences in Lurie's $p$-profinite spaces $\infty$-category are morphisms inducing equivalences on mapping spaces to $p$-finite spaces. There is some group theory involved to reduce the general $p$-finite space case to checking that the map induces isomorphisms on $\mathbb{Z}/p$-valued cohomology. See section 7.3 of \cite{barnea2015pro} for details.
\end{proof}

\begin{proposition}\label{notpprofiniteshapeequivalence}
A morphism of locally connected topoi $f : T \rightarrow T'$ is a $\notp$-profinite shape equivalence if and only if the following three conditions hold,

\begin{enumerate}
\item For any locally constant sheaf $\mathscr{S} \in T'$ of finite sets, we have $H^0(T',\mathscr{S}) \overset{\sim}{\rightarrow} H^0(T,f^* \mathscr{S})$ via the natural map,

\item for any finite group $G$ whose order is not divisible by $p$ with associated constant sheaf $\mathscr{G} \in T'$, we have $\check{H}^1(T',\mathscr{G}) \overset{\sim}{\rightarrow} \check{H}^1(T,f^*\mathscr{G})$ via the natural map,

\item and for each prime $\ell \neq p$ and any locally constant sheaf $\mathscr{L} \in T'$ of finite abelian $\ell$-groups, we have that $H^q(T',\mathscr{L}) \overset{\sim}{\rightarrow} H^q(T,\mathscr{L})$ for all $q \geq 0$ via the natural map.
\end{enumerate}
\end{proposition}

\begin{proof}
The assumptions above imply that the morphism of $\notp$-profinite spaces
\[ \pi_0 \map_{\profssetnotp} (\notpprofiniteshape \EuScript{T}',Z) \simeq \pi_0 \map_{\profssetnotp} (\notpprofiniteshape \EuScript{T},Z) \]
induces an equivalence on connected components when $Z$ is either a $BG$ or a $K(\mathbb{Z}/\ell^n,m)$ for any appropriate $G$, or any $n$ and $m$. We again use the loop-suspension adjunction to conlude that the higher homotopy groups are also equivalent via $f$ and deduce that the full mapping spaces are equivalent for any such $Z$. Since every object in $\ssetnotpfc$ is generated under finite limits by $BG$ and $K(\mathbb{Z}/\ell^n,m)$ spaces, by the last item of Lemma \ref{pro-categorylemma} we are done.
\end{proof}

We wish to study descent of the \'etale homotopy type. To give a general setup of the problem, let $X$ be a scheme and $\{ U_i \rightarrow X\}_{i \in I}$ a faithfully flat and quasi-compact (fpqc) cover of $X$.  We can form the \v{c}ech complex as a simplicial scheme $\check{U}_\bullet$, given by the zeroth coskeleton $\operatorname{cosk}_0^X(\coprod U_i)$. See for example Example 1.3 of \cite{friedlander1982etale} for a definition of coskeletoi.

The universal property of the colimit is that there is a contractible space of maps

\[ \colim_{[n] \in \Delta^{op}} \ettoptype \, \check{U}_n \rightarrow \ettoptype \, X \]

\noindent which makes $X$ a `cocone' over the diagram. If this map is an equivalence, then we say that the \'etale homotopy type satisfies descent for the \v{c}ech cover $\{U_i \rightarrow X \}_{i \in I}$.

More generally we could take a fpqc hypercover $Y_\bullet$ of $X$ and ask if the morphism from the colimit is an equivalence. The case where the schemes are \'etale over $X$ is the `Hypercover Descent' Theorem 3.4 of \cite{isaksen2004etale}. The following theorem formally generalizes Isaksen's Hypercover Descent theorem.

\begin{theorem}
\label{descent-of-shape}
Let $X$ be a category with finite products and fiber products, and let $P \prec Q$ be two Grothendieck topologies, with $P$ finer than $Q$. Write $P/X$ and $Q/X$ for the corresponding sites.

Now assume that for every object $Y \in X$ the natural functor between topoi induces an equivalence on any of the variants of the shapes,

\[ \shape^{\star} P/Y \simeq \shape^{\star} Q/Y \]

where $\shape^{\star}$ is either

\begin{enumerate}
\item the usual shape $\shape$,
\item the protruncated shape $\shape^{\natural}$,
\item the profinite shape $\profiniteshape$,
\item the $\notp$-profinite shape $\notpprofiniteshape$,
\item or the $p$-profinite shape $\pprofiniteshape$.
\end{enumerate}

Then if $Y_\bullet \rightarrow Z$ is a $Q$-hypercover of $Z \in X$, we have descent on the respective $P$-homotopy types,

\[  \shape^{\star} P/Z \simeq \colim_{[n] \in \Delta^{op}} \shape^{\star} P/Y_{[n]} \]

\end{theorem}

We will need the following observation.

\begin{lemma}
Let $Y \rightarrow Z$ be a morphism in $X$. The slice $\infty$-topos $[\hootopos Q/X]/h^Y$ is equivalent to the hypercomplete $\infty$-topos of $Y$, $\hootopos Q/Y$.
\end{lemma}

The proof follows from the universal properties of the slice topos and of the hypercompletion, and so we leave the details to the reader. Armed with the previous lemma, we now tackle the proof of the theorem.

\begin{proof}
Each of the completion operations are left adjoints to materialization, and so they commute with colimits. For brevity, we write $L$ for the relevant completion operation. We have the following equivalence.
\[ \shape^{\star} Q/Z = L\Pi_!(h^Z) \simeq L\Pi_!(\colim_{[n] \in \Delta^{op}} \mathbbm{1}_{Y_{[n]}} ) \]
However $L\Pi_!$ is a left adjoint, the colimit may be brought outside.
\[ L\Pi_!(\colim_{[n] \in \Delta^{op}} \mathbbm{1}_{Y_{[n]}} ) \simeq \colim_{[n] \in \Delta^{op}} L\Pi_!( \mathbbm{1}_Y{_{[n]}} )\]
But of course the right hand side is the colimit of the shapes.
\[ \colim_{[n] \in \Delta^{op}} L\Pi_!( \mathbbm{1}_{Y_{[n]}} ) \simeq \colim_{[n] \in \Delta^{op}} \shape^{\star} Q/Y_{[n]} \]
Collecting the equivalences, we have shown hypercover descent for the $Q$-topology
\[ \shape^{\star} Q/Z  \simeq \colim_{[n] \in \Delta^{op}} \shape^{\star} Q/Y_{[n]} \]
so now we just use the assumptions of the shapes being equivalent,
\[ \shape^{\star} P/Z \simeq \shape^{\star} Q/Z \simeq \colim_{[n] \in \Delta^{op}} \shape^{\star} Q/Y_{[n]} \simeq \colim_{[n] \in \Delta^{op}} \shape^{\star} P/Y_{[n]} \]
and conclude the theorem.
\end{proof}

Assumption $(1)$ of Theorem \ref{descent-of-shape} leads to the strongest conclusion, but does not apply for many topologies arising for algebraic geometry.

\begin{corollary}
The \'etale homotopy type satisfies descent for smooth hypercovers.
\end{corollary}

\begin{proof}
Let $V$ be an arbitrary scheme. We can take $X = Sm/V$ and the two topologies to be the smooth topology and the \'etale topology. Then this follows immediately from Theorem \ref{descent-of-shape}
\end{proof}

\begin{corollary}
The profinitely completed \'etale homotopy type satisfies descent for fppf hypercovers.
\end{corollary}

\begin{proof}
See Lemma 53.82.10 of \cite{stacks-project} for the cohomological comparison. The representability in schemes of finite torsors in both categories is a consequence of Theorem III.4.3 of \cite{milne-etalebook}.
\end{proof}

Finally we end with a pair of purely topological lemmas. First we need a definition, compare to Definition 8.2.7 of \cite{huber2013etale}.

\begin{definition}
Let $X$ be a topological space. A sheaf $\mathscr{F}$ of sets or $\infty$-groupoids on $X$ is called \emph{overconvergent} if for every pair of fiber functors $x^*,y^* : \topos X \rightarrow Set$ and every natural transformation $x^* \leadsto y^*$ the natural transformation gives an equivalence

\[ x^* \mathscr{F} \simeq y^* \mathscr{F} \]
\end{definition}

\begin{lemma}\label{lcsareoverconvergent}
Let $X$ be a locally spectral topological space. Then every locally constant sheaf $\mathscr{F}$ of sets or $\infty$-groupoids is overconvergent.
\end{lemma}

\begin{proof}
The underlying set of $X$ provides its topos with enough points. We prove this for a specialization of points $g \leadsto s$. Choose a neighborhood $U$ of $s$ on which $\mathscr{F}$ is constant. Then $U$ is also a neighborhood of $g$ on which $\mathscr{F}$ is constant, and so we see that $\mathscr{F}_s \rightarrow \mathscr{F}_g$ is an equivalence. Since all morphisms of fiber functors are generalizations, we are done.
\end{proof}

\begin{corollary}\label{overconvergentcorollary}
Assume that the full sub-category of hypercomplete overconvergent sheaves of $\infty$-groupoids on a locally spectral topological space $X$ forms an $\infty$-topos, $\EuScript{T}_{oc}$, such that the inclusion $i:\EuScript{T}_{oc} \subset \EuScript{T}$ is the left adjoint of a geometric morphism. Then the shape of the hypercomplete $\infty$-topos $\EuScript{T}$ of $X$ is equivalent to the shape of $\EuScript{T}_{oc}$.
\end{corollary}

\begin{proof}
Call the geometric morphism $\nu$, so that $\nu^* = i$ . Since $\nu^*$ is fully faithful, a standard mapping space argument shows that that $\nu_* \nu^*$ is equivalent to the identity functor via the unit morphism. This says the functors $\pi_*\pi^*$ and $\pi_* \nu_* \nu^* \pi^*$ are equivalent, which is exactly the desired shape equivalence.
\end{proof}

We end with the following lemma.

\begin{lemma}\label{overconvergentsheavesareconstantirred}
The $\infty$-category of hypercomplete overconvergent sheaves on an irreducible sober space is equivalent to $\sset$ via the global sections functor.
\end{lemma}

\begin{proof}
Name the space $X$ and let $\mathscr{F}$ be an overconvergent sheaf of $\infty$-groupoids on $X$. An irreducible sober space has a generic point $x \in X$. We wish to show that $\mathscr{F}$ is equivalent to the constant sheaf associated with the space $\mathscr{F}_x$. Call the constant sheaf associated with the space $\mathscr{F}_x$ by $\mathscr{G}$.

We simply need to build a morphism $\mathscr{G} \to \mathscr{F}$ that induces an equivalence on the stalk at $x$. That both sheaves are overconvergent immediately implies that it is an equivalence at every other point. Hypercompleteness implies that a stalkwise equivalence is an equivalence.

The morphism is actually just the unit of the adjunction given by the continuous map $i : x \to X$, that is we can identify $i_* i^* \mathscr{F} = \mathscr{G}$. Note that for such an inclusion, both $i_*$ and $i^*$ preserve weak equivalences and so this models the adjunction of $\infty$-categories. Since all opens in an irreducible sober space are connected, we deduce that the stalks at $x$ of $\mathscr{F}$ and $\mathscr{G}$ are equivalent, and conclude the statement of the lemma.
\end{proof}

\section{Etale homotopy type of non-archimedean analytic spaces}

Let $K$ be a non-archimedean field, and $\mathcal{X}$ a non-archimedean analytic space over $K$. Then in \cite{berkovich1993etale} Berkovich develops an \'{e}tale theory and in \cite{berkovich1994vanishing} a quasi-\'etale theory for such geometric objects. We will distinguish the two and remark on the cases in which they are equivalent.

\begin{definition}
For such $\mathcal{X}$, put $\site \mathcal{X}_{\acute{e}t}$ for the small \'etale site of $\mathcal{X}$, put $\topos \mathcal{X}_{\acute{e}t}$ for the \'etale 1-topos of $X$, and $\hootopos \mathcal{X}_{\acute{e}t}$ for the hypercomplete \'etale $\infty$-topos. Similarly, we write $\site \mathcal{X}_{q\acute{e}t}$, $\topos \mathcal{X}_{q\acute{e}t}$, and $\hootopos \mathcal{X}_{q\acute{e}t}$ for the quasi-\'etale variants of the previous. We will now define several objects of interest, letting $\ell$ be a prime number.

\begin{enumerate}
\item We write $\ettoptype \mathcal{X}$, for the shape of the $\infty$-topos $\hootopos \mathcal{X}_{\acute{e}t}$, and write $\ettoptype^{\natural} \mathcal{X}$ for its protruncation, $\profettoptype \mathcal{X}$ for its profinite completion, $\profettoptype_{\notp} \mathcal{X}$ for its $\notp$-profinite completion, and $\profettoptype_\ell \mathcal{X}$ for its $\ell$-profinite completion, 

\item We write $\qettoptype \mathcal{X}$ for the shape of $\hootopos \mathcal{X}_{q\acute{e}t}$, and write $\qettoptype^{\natural} \mathcal{X}$ for its protruncation, $\profqettoptype \mathcal{X}$ for its profinite completion, $\profqettoptype_{\notp} \mathcal{X}$ for its $\notp$-profinite completion, and $\profqettoptype_\ell \mathcal{X}$ for its $\ell$-profinite completion, 

\end{enumerate}
\end{definition}

In section 3 of \cite{berkovich1994vanishing}, Berkovich defines a class of maps called quasi-\'etale. The definition of a quasi-\'etale map of non-archimedean analytic spaces is similar to the definition of the $G$-topology, in that its relation to the \'etale topology is similar to the relation of the $G$-topology to the hausdorff topology. In this theory, \'etale maps are open, whereas the affinoid domains are topologically compact and thus closed. For this definition of \'etale, the affinoid domains are \emph{not} necessarily \'etale. The quasi-\'etale allows these maps to also be coverings, slightly shrinking the overall topos without affecting the topological invariants.

\begin{conjecture}
Let $\mathcal{X}$ be a $K$-analytic space. The identity functor $\mu_\mathcal{X} : \site \mathcal{X}_{\acute{e}t} \rightarrow \site \mathcal{X}_{q\acute{e}t}$ which sends an \'etale map to itself in the quasi-\'etale site induces an equivalence on the protruncated shapes of the $\infty$-topoi.

\[  \mu_{\mathcal{X},!}: \qettoptype \mathcal{X} \overset{\sim}{\rightarrow} \ettoptype \mathcal{X} \]
\end{conjecture}

\begin{remark}
We can almost prove the equivalence of protruncated shapes, that is if we add $\natural$ to the shapes above.

We can apply Theorem 3.3 (ii) for $f = \mathrm{id}_\mathcal{X}$ of \cite{berkovich1994vanishing} to obtain the cohomological comparison we need. By Corollary 3.5 of \cite{berkovich1994vanishing}, we know that the category of locally constant sheaves of sets on $\mathcal{X}_{\acute{e}t}$ embeds fully faithfully into the category of locally constant sheaves of sets on $\mathcal{X}_{q\acute{e}t}$.

We simply need essential surjectivity. Let $\mathscr{G}$ be a quasi-\'etale $G$-torsor for some group $G$. Let $\{U_i' \to \mathcal{X}\}$ be a quasi-\'etale cover on which $\mathscr{G}$ is trivialized. Refine it to a cover $\{ U_i \to \mathcal{X} \}$ with each $U_i$ affinoid, and choose a factorization $U_i \to W_i \to \mathcal{X}$ so that the first map is an affinoid embedding and the second is \'etale. Then $\mathscr{G}$ is of course still trivialized on this new cover. Let it be trivialized by some family of sections we'll denote $s_i : U_i \to \mathscr{G}|_{U_i}$. If we knew the representability of quasi-\'etale torsors, we would be able to extend these sections to some neighborhood of each $U_i$ in $W_i$, see for example the proof of Theorem 3.3 in \cite{berkovich1994vanishing}. This would show that the space is \'etale locally a trivial $G$-torsor.

Unfortunately the author does not know of any such representability result for quasi-\'etale torsors. We will deduce this equivalence in a special case by comparing to the theory of adic spaces.
\end{remark}

There are other formalisms of non-archimedean analytic geometry with their corresponding classes of \'etale morphisms. We can obtain a quick comparison theorem with that of rigid analytic geometry in the case where $\mathcal{X}$ is paracompact hausdorff and strictly $K$-analytic.

\begin{theorem}
\label{comparison-berkovich-adic}
Let $\mathcal{X}$ be a hausdorff strictly $K$-analytic space over a complete non-archimedean field $K$, write $\mathcal{X}^{ad}$ for the corresponding adic space.

Then,
\begin{enumerate}
\item the quasi-\'etale 1-topos of $\mathcal{X}$ is equivalent to the \'etale 1-topos of $\mathcal{X}^{ad}$

\item the quasi-\'etale $\infty$-topos of $\mathcal{X}$ is equivalent to the \'etale $\infty$-topos of $\mathcal{X}^{ad}$,

\item and the hypercomplete quasi-\'etale $\infty$-topos of $\mathcal{X}$ is equivalent to the hypercomplete \'etale $\infty$-topos of $\mathcal{X}^{ad}$
\end{enumerate}
\end{theorem}

\begin{remark}
The \'etale \emph{site} of $\mathcal{X}$ is equivalent to the partially proper \'etale site of $\mathcal{X}^{ad}$. We see this by combining Proposition 8.3.4 of \cite{huber2013etale}, which identifies the partially proper \'etale site of the corresponding rigid space $\mathcal{X}^{rig}$ with the \'etale site of $\mathcal{X}$, and the discussion in 8.2.11 of \cite{huber2013etale} which identifies the categories and site-theoretic coverages of partially proper \'etale rigid morphisms to $\mathcal{X}^{rig}$ to those of $\mathcal{X}^{ad}$.

We use the same ideas along with the arguments from section 8.1 of \cite{huber2013etale} to prove this theorem.
\end{remark}

\begin{proof}
The quasi-\'etale topos is equivalent to the one generated by quasi-\'etale morphisms whose source is affinoid. We use that $\mathcal{X}$ is hausdorff to have fiber products in the subcategory. Such a quasi-\'etale map $q : U \to \mathcal{X}$ will factor into $q = e \circ j$ where $e$ is \'etale and $j$ is an affinoid embedding. Both $e^{ad}$ and $j^{ad}$ will be \'etale as morphisms of adic spaces, and $e^{ad}$ is also partially proper. The composition in general is merely \'etale and not partially proper and \'etale.

We will show that such morphisms are dense in the category of \'etale maps to $\mathcal{X}^{ad}$. Let $f : Y \rightarrow \mathcal{X}^{ad}$ be an \'etale map, and we may assume without loss of generality that $Y = \spa A$ is affinoid and its image is contained in a partially proper affinoid of $\mathcal{X}^{ad}$ as such morphisms generate a cofinal system of coverings. Theorem 2.2.8 of \cite{huber2013etale} gives a local compactification $f = \overline{f} \circ j$ where $\overline{f}$ is a finite \'etale map and $j$ is an open embedding. Replace our affinoid cover by the local and compactifiable one. The finite \'etale map of course comes from a finite \'etale extension of the corresponding analytic domain in $\mathcal{X}$, and the open embedding may not directly come from a affinoid in $\mathcal{X}^{rig}$, however it does refine to a collection of affinoids in $\mathcal{X}^{rig}$ whose union is $\spa A$. To relate this back to $\mathcal{X}$ itself, we apply Theorem 1.6.2 of \cite{berkovich1993etale}.
\end{proof}

The reader should beware that the above gives a functorial equivalence of topoi, however there are subtleties to when a morphism $f: \mathcal{X} \rightarrow \mathcal{Y}$ induces an equivalence of functors between $\infty$-topoi $\theta_{\mathcal{Y}}^*f_* \simeq f^{ad}_* \theta_{\mathcal{X}}^*$ see Theorem 8.3.5 of \cite{huber2013etale} for some criteria for abelian sheaves.

\begin{corollary}\label{comparison-qet-et}
The quasi-\'etale homotopy type of $\mathcal{X}$ is equivalent to the \'etale homotopy type of $\mathcal{X}$ at least when $\mathcal{X}$ is hausdorff and strictly $K$-analytic.
\end{corollary}

\begin{proof}
We already know that the \'etale 1-topos is the subcategory of overconvergent sheaves in the quasi-\'etale 1-topos. We need to argue that the hypercomplete \'etale $\infty$-topos is the sub-$\infty$-category of hypercomplete overconvergent sheaves in the quasi-\'etale $\infty$-topos

Examining the proof of Theorem 8.2.6.i.a of \cite{huber2013etale}, we really only need the statements (I) through (VI). In fact (I) through (V) are geometric in nature, and happen on the level of sites. This means that we really only need to extend statement (VI), which says that overconvergent sheaves on an irreducible topological space are acyclic. The rest of Huber's argument is about understanding the underlying site, and does not directly interact with the $\infty$-categorical machinery. We simply apply Lemma \ref{overconvergentsheavesareconstantirred}, and conclude the corollary.
\end{proof}

\begin{remark}
In fact, roughly speaking Huber's argument can be translated into the following. Write $X$ for an analytic adic space.

Both the \'etale and the \'etale-partially-proper hypercomplete $\infty$-topoi satisfy descent for \'etale-partially-proper hypercovers, so we deduce that as `cosheaves of $\infty$-topoi' on the hypercomplete \'etale-partially-proper site that
\[ \hootopos X_{\acute{e}t} \simeq \colim_{U_\bullet \in HR(X_{\acute{e}t.p.p.})} \hootopos U_\bullet^{\acute{e}t} \simeq \colim_{U_\bullet \in HR(X_{\acute{e}t.p.p.})} \hootopos U_\bullet \]
The first equivalence is essentially a tautology. The second equivalence can be checked stalkwise, once we know that the stalks are what Huber calls $Z(\overline{z})$ in for example statement (V) of the proof of Theorem 8.2.6.i.a of \cite{huber2013etale}, and that the \'etale topos of such $Z(\overline{z})$ is equivalent to its zariski topos.

Similarly, we have
\[ \hootopos X_{\acute{e}t.p.p.} \simeq \colim_{U_\bullet \in HR(X_{\acute{e}t.p.p.})} \hootopos U_\bullet^{\acute{e}t.p.p.} \simeq \colim_{U_\bullet \in HR(X_{\acute{e}t.p.p.})} \hootopos U_\bullet^{p.p.} \]
We have the `obvious' map
\[ \colim_{U_\bullet \in HR(X_{\acute{e}t.p.p.})} \hootopos U_\bullet\to \colim_{U_\bullet \in HR(X_{\acute{e}t.p.p.})} \hootopos U_\bullet^{p.p.} \]
We can take the stalk of both objects which is
\[ \Big( \colim_{U_\bullet \in HR(X_{\acute{e}t.p.p.})} \hootopos U_\bullet \Big)_x \to \Big(\colim_{U_\bullet \in HR(X_{\acute{e}t.p.p.})} \hootopos U_\bullet^{p.p.}\Big)_x \]
Since taking stalks commutes with colimits the above is really just
\[ \hootopos Z(\overline{x}) \simeq \colim_{x \in \operatorname{im} U} \hootopos U \to \colim_{x \in \operatorname{im} U} \hootopos U_{p.p.} \simeq \hootopos Z(\overline{x})_{p.p.} \]
The shape functor commutes with colimits, and so we can functorially identify the stalks at $\overline{x}$ of the morphism of sheaves
\[ \shape \hootopos X_{\acute{e}t} \rightarrow \shape \hootopos X_{\acute{e}t.p.p.}\]
as being
\[  \shape \hootopos Z(\overline{x}) \rightarrow \shape \hootopos Z(\overline{x})_{p.p.}\]
And then we conclude that the morphism is an equivalence on stalks by Corollary \ref{overconvergentcorollary}.

It seems plausible that the above argument could be reworked using reified valuative spectra of \cite{kedlaya-reified} to conclude the above for arbitrary hausdorff $K$-analytic spaces. Since this is outside of the scope of this document, we set aside the problem for future work. For the next section, we only need the equivalence of protruncated shapes between the quasi-\'etale and the \'etale topoi for strictly $K$-analytic varieties.
\end{remark}

As a tool to give counter-examples later, we now characterize the rational \'etale cohomology of a non-archimedean analytic space. The following proposition is presumably well known, and appears as a comment in \cite{berkovichoneformsbook}. We give the easy proof for completeness.

\begin{proposition}\label{rationalchangeoftopssdegen}
For a non-archimedean analytic space $\mathcal{X}$, we have a natural isomorphism

\[ H^i(\mathcal{X}_{q\acute{e}t},\mathbb{Q}) \cong H^i(\mathcal{X}_{\acute{e}t},\mathbb{Q}) \cong H^i(|\mathcal{X}|,\mathbb{Q}) \]
\end{proposition}

\begin{proof}
Recall, following the notation of Section 4.2 of \cite{berkovich1993etale} that we have a morphism of sites $\pi : \mathcal{X} \rightarrow |\mathcal{X}|$. This leads to a spectral sequence
\[ E_2^{p,q} = H^q(|\mathcal{X}|,R^p\pi_* \mathbb{Q})\]
Now we must examine the sheaves $R^p\pi_* \mathbb{Q}$. Consider the stalk at a point $x \in |\mathcal{X}|$ and the stalk's description as Galois cohomology from the same section of \cite{berkovich1993etale},
\[ (R^p\pi_* \mathbb{Q})_x = H^p(\operatorname{Gal}\mathscr{H}(x), \mathbb{Q}) \]
But the right hand side is trivial for $p >0$. Thus the above spectral sequence degenerates at $E_2$ and we get the desired isomorphism.
\end{proof}

Now let $K$ be a complete non-archimedean field, with non-trivial valuation, and let $V$ be a variety over $K$, \emph{i.e.} an integral scheme of finite type over $K$. The variety $V$ has an \'etale homotopy type, and we have just defined an \'etale homotopy type for the analytification $V^{an}$. We can immediately deduce some comparison theorems from the literature.

Recall Theorem 3.1 of \cite{berkovich1995comparison}, which we reproduce in our notation below.

\begin{theorem}[\cite{berkovich1995comparison}]
Let $\phi : Y \rightarrow X$ be a morphism of finite type between schemes which are locally of finite type over $K$, and let $\mathscr{F}$ be a constructible abelian sheaf on $Y$ of torsion groups whose orders are coprime to the characteristic of $K$. Then for any $q \geq 0$ there is a canonical isomorphism

\[ (R^q \phi_* \mathscr{F})^{an} \overset{\sim}{\rightarrow} R^q\phi_*^{an}\mathscr{F}^{an} \]
\end{theorem}

In the non-proper case, we must pass to the $\notp$-profinitely completed homotopy type.

\begin{theorem}\label{artin-comparison-ell}
Let $p$ be the characteristic of $K$. If $p>0$, then the $\notp$-profinitely completed \'etale homotopy type of the variety $V$ is equivalent to the $\notp$-profinitely completed \'etale homotopy type of its analytification $V^{an}$.
\[ \profettoptype_\notp V \simeq \profettoptype_\notp V^{an} \]

If $p = 0$ or $V$ is proper, then it is true for the profinitely completed \'etale homotopy types.

\[ \profettoptype V \simeq \profettoptype V^{an} \]
\end{theorem}

\begin{proof}
We wish to apply Proposition \ref{notpprofiniteshapeequivalence}. By the referenced theorem of \cite{berkovich1995comparison}, we have the needed comparison result for abelian cohomology. The remaining question is on fundamental groups, which in the case that $K$ is of characteristic zero is a direct corollary of Theorem 3.1 of \cite{lutkebohmert1993riemann}, and in the case that $K$ is of characteristic $p>0$, follows from Theorem 4.1 of \cite{lutkebohmert1993riemann} which gives an equivalence on the category of $G$-torsors for any group $G$ whose order is not divisible by $p$.

Now assume that $V$ is proper.

In both the classical and Berkovich \'etale categories, sheaf theoretic $G$-torsors are representable as spaces finite and \'etale over their respective bases. These spaces are in bijection with locally free commutative algebra objects in the category coherent modules, and GAGA gives an equivalence between the two categories. This proves the two profinite spaces have the same fundamental group for any basepoint.

To compare the $p$-primary abelian cohomology, we will use compactly supported \'etale cohomology. For a general proper morphism $f: \mathcal{X} \rightarrow \mathcal{Y}$ of non-archimedean analytic spaces and a constructible sheaf $\mathscr{F}$, Berkovich's $f_! \mathscr{F}$ is not equal to $f_* \mathscr{F}$. In the case where $\mathcal{Y}$ is compact, their global sections \emph{are} equal. This implies that at least the derived functors
\[ R\Gamma \circ Rf_! \simeq R\Gamma \circ Rf_* \]
are equal. In particular if $\mathcal{Y}$ is a point and $\mathcal{X}$ is compact, we see that the compactly supported \'etale cohomology agrees with the usual \'etale cohomology
\[ H^i_c(\mathcal{X}_{\acute{e}t}, \mathscr{F}) \cong H^i(\mathcal{X}_{\acute{e}t}, \mathscr{F}) \]
We can then apply the comparison theorem for compactly supported cohomology of \cite{berkovich1993etale}, which does not have the restriction on the orders of the coefficient groups. This gives the chain of equivalences

\[ H^i(V_{\acute{e}t},\mathscr{F}) \cong H^i_c(V_{\acute{e}t},\mathscr{F}) \cong H^i_c(V^{an}_{\acute{e}t},\mathscr{F}) \cong H^i(V^{an}_{\acute{e}t},\mathscr{F}) \]
\end{proof}

In fact the above point that principal homogeneous spaces are representable in the category of non-archimedean $K$-analytic spaces does not require finiteness hypotheses. This proves the following theorem.

\begin{theorem}
Let $\mathcal{X}$ be a non-archimedean $K$-analytic space and let $\overline{x}$ be a given geometric point of $\mathcal{X}$, then the fundamental pro-group of the \'etale homotopy type of $\mathcal{X}$ is isomorphic to de Jong's \'etale fundamental group as defined in \cite{dejongetalefundamentalgroups}.

\[  \pi_1^{top}(\ettoptype \mathcal{X}, \overline{x}) \cong \pi_1^{\acute{e}t}(\mathcal{X},\overline{x})\]
\end{theorem}

On the left hand side, the pair $(\ettoptype \mathcal{X},\overline{x})$ means we are identifying the pro-space $\ettoptype \overline{x}$ as a point, and functoriality gives the morphism $\ettoptype \overline{x} \rightarrow \ettoptype \mathcal{X}$ which witnesses a point on the latter pro-space.

A natural question to ask would be if the above shape equivalence holds without profinite completion. This fails almost immediately, as de Jong's observes in \cite{dejongetalefundamentalgroups} that the $p$-adic logarithm is a $\cup_i \mu_{p^i} $-torsor over $\mathbb{A}^1_{\mathbb{C}_p}$. A natural solution would be to assume properness, but this of course does not work -- even in the discretely valued case -- as the following example shows.

\begin{example}
Let $E$ be an elliptic curve over a complete discretely valued field $K$, which admits a regular but not smooth semi-stable model over the valuation ring $R$. Then the map of \'etale homotopy types is not an equivalence of protruncated spaces
\[ \ettoptype E^{an} \rightarrow \ettoptype E \]
To see this, just compute $H^1(-,\mathbb{Q})$ using Theorem \ref{rationalchangeoftopssdegen}.
\end{example}

In fact in \cite{dejongetalefundamentalgroups} it is even shown that $\mathbb{P}^n_{\mathbb{C}_p}$ itself has interesting fundamental group.

Similarly, we can ask about analogues of Friedlander's homotopy fiber theorem, see \cite{friedlander1982etale}. We can modify the above to see that it must be false without $\ell$-profinite completion, even in pure characteristic zero.

\begin{example}
Give $\mathbb{A}^2 \times \mathbb{P}^2$ the coordinates $(a,b; \, x:y:z)$, and consider the space $\mathcal{Y} \subset \mathbb{A}^4$ defined by the ideal $(y^2z - x^3 - axz^2 - bz^3)$. The map $\pi : \mathcal{Y} \rightarrow \mathbb{A}^2$ is clearly projective, but has some non-smooth fibers. Remove the discriminant locus $\Delta$ in $\mathbb{A}^2$ to get $\mathcal{B} = \mathbb{A}^2 \setminus \Delta$, and write $\mathcal{E} = \pi^{-1} \mathcal{B}$.

Then $\pi : \mathcal{E} \to \mathcal{B}$ is smooth and projective, but different fibers may have different \'etale homotopy types. Since every elliptic curve up to isomorphism is a fiber in this family, both elliptic curves of good reduction and semi-stable bad reduction appears as a fiber in the family. In particular, the \'etale homotopy type of at least one of the fibers is not equivalent to the homotopy fiber of $\ettoptype \mathcal{E} \rightarrow \ettoptype \mathbb{A}^1_K$.
\end{example}

Since analytification preserves fiber products, we know that $(X_{\overline{y}})^{an} = X^{an}_{\overline{y}^{an}}$. Theorem \ref{artin-comparison-ell} implies that if we analytify one of these fiber sequences it remains a fiber sequence. In fact Friedlander's proof works in the non-archimedean world, under somewhat restrictive assumptions.

\begin{proposition}
Let $f : \mathcal{X} \rightarrow \mathcal{Y}$ be a smooth and proper map of non-archimedean $K$-analytic spaces for $K$ non-trivially valued. Assume that

\begin{enumerate}
\item both are compact and hausdorff,

\item the morphism $f$ has geometrically connected fibers,

\item the \'etale fundamental group $\pi_1(\mathcal{Y})$ acts trivially on the cohomology of the fibers,

\item and that the analytic space $\mathcal{Y}$ is connected.
\end{enumerate}

Then the following is a homotopy fiber sequence of $\ell$-profinite spaces.

\[ \profettoptype_\ell \mathcal{X}_{\overline{y}} \rightarrow \profettoptype_\ell \mathcal{X} \rightarrow \profettoptype_\ell \mathcal{Y}  \]
\end{proposition}

\begin{proof}
We have three converging spectral sequences and morphisms between them
\begin{center}
\begin{tikzcd}[column sep=small]
E_2^{p,q}(S) & = & H^p_{\mathrm{sing}}(\profettoptype_\ell \mathcal{Y}, H^q(F,\mathbb{Z}/\ell)) & \Rightarrow & H^{p+q}_{\mathrm{sing}}(\profettoptype_\ell \mathcal{X},\mathbb{Z}/\ell) \\
E_2^{p,q}(S') \arrow[u] & =  & H^p_{\mathrm{sing}}(\profettoptype_\ell \mathcal{Y}, H^q(\mathcal{X}_{\overline{y}},\mathbb{Z}/\ell)) \arrow[u] & \Rightarrow & H^{p+q}_{\mathrm{sing}}(\profettoptype_\ell \mathcal{X},\mathbb{Z}/\ell) \arrow[u] \\
E_2^{p,q}(V) \arrow[d] \arrow[u, "\simeq"] & = & H^p_{\mathrm{sheaf}}(\mathcal{Y}_{\acute{e}t}, \underline{ H^q(X_{\overline{y}}, \mathbb{Z}/\ell)}) \arrow[d] \arrow[u, "\simeq"] & \Rightarrow &  H^{p+q}_{\mathrm{sheaf}}(\mathcal{X}_{\acute{e}t}, \mathbb{Z}/\ell) \arrow[d] \arrow[u, "\simeq"]\\
E_2^{p,q}(L) & = & H^p_{\mathrm{sheaf}}(\mathcal{Y}_{\acute{e}t},R^qf_* \mathbb{Z}/\ell) & \Rightarrow & H^{p+q}_{\mathrm{sheaf}}(\mathcal{X}_{\acute{e}t},\mathbb{Z}/\ell)
\end{tikzcd}
\end{center}
The marked isomorphisms come from Verdier's hypercover theorem. If each $R^q f_* \mathbb{Z}/\ell$ is constant, then $E_2^{p,q}(V) \to E_2^{p,q}(L)$ is also an isomorphism. The map from $E_2^{p,q}(S')$ to $E_2^{p,q}(S)$ is the one induced by the universal property of the homotopy fiber. For $R^q f_* \mathbb{Z}/\ell$ to be constant, it is enough to know that $R^if_*$ takes locally constant $\ell$-torsion sheaves to locally constant $\ell$-torsion sheaves, \emph{i.e.} proper-smooth base change. If $\mathcal{X}$ and $\mathcal{Y}$ are strictly $K$-analytic, we may appeal to the theory of adic spaces to show this. See for example Theorem 6.2.2 of \cite{huber2013etale} and its Corollary, noting that Propositions 3.3.1 and 3.3.2 of \cite{berkovich2012spectral} guarantee the associated morphism of rigid spaces is smooth and proper, and Remark 1.3.19 and Proposition 1.7.11 of \cite{huber2013etale} guarantee the associated morphism of adic spaces is smooth and proper. Vladimir Berkovich has pointed out in private communication that one may tensor up to a larger non-archimedean field where the base space is strictly analytic, and then push the results back down to deduce that the higher direct images of locally constant sheaves are locally constant for a smooth and proper morphism with non-strictly analytic base.

From here, we simply proceed as in \cite{friedlander1982etale}. We create a tower of \'etale coverings of $\mathcal{Y}$, and in the colimit the cohomology spectral sequences all collapse. The isomorphism on abutments forces the $E_2^{p,q}$ terms to become isomorphic, so that in particular
\[ H^q(\mathcal{X}_{\overline{y}},\mathbb{Z}/\ell) \to H^q(F,\mathbb{Z}/\ell) \]
is an isomorphism. In $\ell$-profinite spaces, this says exactly that they are equivalent.
\end{proof}

\section{Kummer \'etale homotopy type of log formal schemes}

Throughout this section, we will let $R$ be a complete discrete valuation ring, $\mathfrak{m}=(\pi)$ its maximal ideal with chosen non-zero uniformizer $\pi$, $K$ its field of fractions, and $k$ its residue field.

We must first review the theory already in place for formal schemes, set up in \cite{berkovich1994vanishing} and \cite{berkovich1996vanishing}.

\begin{definition}
A formal scheme $\mathfrak{X}/\spf R$ is called \emph{special formal} if it locally is the formal spectrum of an adic ring $A$, and for some ideal of definition $\mathfrak{a} \subset A$, the quotients $A/\mathfrak{a}^n$ for positive integer powers are all finitely generated over $R$.
\end{definition}

Any formal scheme topologically of finite type is special. Recall that the local models for special formal schemes are formal affines that admit a closed immersion into the formal spectrum of the topological ring $R\{T_1,\ldots,T_m\}[[S_1,\ldots,S_n]]$.

\begin{definition}
Let $f: \mathfrak{X} \rightarrow \mathfrak{Y}$ a morphism of special formal scheme over $\spf R$. We say $f$ is \emph{adically \'etale} if it is locally topologically of finite type, and for any ideal of definition $I$ of $\mathfrak{Y}$, the induced morphism of schemes $(\mathfrak{X},\mathcal{O}_\mathfrak{X}/I\mathcal{O}_\mathfrak{X}) \rightarrow (\mathfrak{Y},\mathcal{O}_\mathfrak{Y}/I)$ is classically \'etale.
\end{definition}

These morphisms do not obviously give a topos of direct interest from the perspective of non-archimedean geometry. Proposition 2.1 of \cite{berkovich1996vanishing} gives a natural equivalence between the adically \'etale topos of $\mathfrak{X}$ and the classical \'etale topos of its closed fiber $\mathfrak{X}_s$. In particular the adically \'etale topos of a special formal scheme is not invariant under admissible blow up. We can, however, use this topology to define log structures on such schemes.

\begin{definition}
Let $\mathfrak{X}$ be a special formal scheme. A \emph{pre-log structure} on $\mathfrak{X}$ is the data $(\mathcal{M}_\mathfrak{X},\alpha)$ of a sheaf of commutative monoids $\mathcal{M}_\mathfrak{X}$ for Berkovich's adically \'etale topology, and a homomorphism of sheaves of monoids $\alpha : \mathcal{M}_\mathfrak{X} \rightarrow \mathcal{O}_\mathfrak{X}$, where the target sheaf is given the multiplicative structure.

In the case where $\alpha: \alpha^{-1}(\mathcal{O}_\mathfrak{X}^* )\rightarrow \mathcal{O}_\mathfrak{X}$ is an isomorphism, we say that $(\mathcal{M}_\mathfrak{X},\alpha)$ is a \emph{log structure} on $\mathfrak{X}$.

Given any pre-log structure $\mathcal{M}'$, there is a unique up to unique isomorphism associated log structure $\mathcal{M}$ through which any morphism $\mathcal{M}' \rightarrow \mathcal{N}$ to a log structure factors uniquely. This is given by the pushout of

\[ \mathcal{M}' \leftarrow \alpha^{-1}\mathcal{O}_X \rightarrow \mathcal{O}_X \]

A \emph{morphism} $(f,f^\#): (\mathfrak{X},\mathcal{M}, \alpha) \rightarrow (\mathfrak{Y}, \mathcal{N}, \beta)$ \emph{of log special formal schemes} is a morphism of underlying special formal schemes $f$, along with a compatibility morphism $f^\# : f^{-1}\mathcal{N} \rightarrow \mathcal{M}$ of the log structures, in the sense that it gives the obvious commutative diagram between $f^{-1} \beta : f^{-1}\mathcal{N} \rightarrow f^{-1}\mathcal{O}_\mathfrak{Y}$ and $\alpha : \mathcal{M} \rightarrow \mathcal{O}_\mathfrak{X}$.
\end{definition}

The primary case of interest is where the log structure is given by the functions vanishing only along the special fiber. In fact we will give all special formal schemes the \emph{canonical log structure}, denoted $\mathcal{M}_\mathfrak{X}^{can}$ which is given by setting the sheafs sections $\mathcal{M}_\mathfrak{X}^{can}(\mathfrak{Y})$ on an adically \'etale morphism $\mathfrak{Y} \rightarrow \mathfrak{X}$ to be the subset of $\mathcal{O}_\mathfrak{Y}$ given functions which are invertible on the non-archimedean analytic space associated to the generic fiber $\mathfrak{Y}_\eta$. If $Y \subset \mathfrak{X}_s$ is a subscheme, then we can complete along it to get $\hat{\mathfrak{X}}_{Y}$. This can be given a log structure by demanding the morphism $\hat{\mathfrak{X}}_Y\rightarrow \mathfrak{X}$ be a strict log morphism, that is by letting the log structure on $\mathfrak{X}$ generate one on $\hat{\mathfrak{X}}_Y$. This does not agree in general with the canonical log-structure on $\hat{\mathfrak{X}}_Y$.

We can generalize this type of log structure with the notion of a vertical log structure from \cite{berkovichcomplex}

\begin{definition}
The log structure $\mathcal{M}_\mathfrak{X}$ on a log formal scheme will be called \emph{vertical} if the localization of $\mathcal{M}_\mathfrak{X}$ with respect to the sub-monoid $R\setminus \{0\}$ is a sheaf of abelian groups.
\end{definition}

An important example is that the canonical log structure is vertical.

Working with \'etale sheaves is somewhat unwieldy, so we bring along the notion of a chart for a log structure.

\begin{definition}
A log special formal scheme $(\mathfrak{X}, \mathcal{M})$ will be said to have a \emph{chart} $\mathcal{P} \rightarrow \mathcal{M}$ if $\mathcal{P}$ is a constant sheaf in the adically \'etale topology for a monoid $P$, and $\mathcal{M}$ is the associated log structure.

A log structure on a special formal scheme is said to be \emph{coherent} if it adically \'etale locally admits charts. The log structure will be called \emph{fine} or \emph{saturated} if all such charts are fine or saturated. See for example \cite{kato1989logarithmic} for full definitions.
\end{definition}

\subsection{Algebra of log special formal schemes}

There are two natural candidates for classes of `good' morphisms between special formal schemes, those that are completions of morphisms of schemes finite type and flat over $\spec R$ with the corresponding property, and those that are `adically' so. For our purposes it turns out the base objects are those of the former type, but the Grothendieck topology will consist of morphisms of the latter type.

\begin{definition}
Let $\mathfrak{X}$ be a morphism  between fine vertical log special formal schemes. We will say that:
\begin{enumerate}
\item $f$ is algebraizably formally $R$-log smooth,
\item $f$ is algebraizably $R$-log smooth
\end{enumerate}
if adically \'etale locally we can realize the morphism as the completion of a morphism of flat and finite type fine log schemes over $\spec R$. That is, as a map $\phi : (X,M) \rightarrow \spec R$ completed along a locally closed subscheme $Z \subset \mathfrak{X}_s$, e.g.
\[ \hat{X}_{Z} \rightarrow \hat{X} \rightarrow \spf R \]
and further more $\phi$ satisfies the corresponding property in the following list:

\begin{enumerate}
\item if $\phi$ can be taken to be log smooth on a neighborhood of $Z$,

\item if $\phi$ can be taken to be log smooth on a neighborhood of $Z$, and $Z = X_s$,
\end{enumerate}

This is the same notion as \emph{formally $R$-log smooth} and \emph{$R$-log smooth} of \cite{berkovichcomplex}.
\end{definition}

The above can be made relative to a morphism of log special formal schemes, but it does not seem to result in anything interesting in the generality we are working in.

\begin{definition}
Let $f : \mathfrak{X} \rightarrow \mathfrak{Y}$ be a morphism which is topologically of finite type between fine vertical log special formal schemes. We will say that:
\begin{enumerate}
\item $f$ is adically log \'etale
\item $f$ is adically kummer \'etale
\end{enumerate}
if they are adically so, that is when for any ideal of definition $\mathfrak{I}$ of $\mathfrak{Y}$ the induced morphism $\overline{f} : (\mathfrak{X}, \mathcal{O}_\mathfrak{X}/\mathfrak{I}\mathcal{O}_\mathfrak{X}) \rightarrow (\mathfrak{Y}, \mathcal{O}_\mathfrak{Y}/\mathfrak{I})$ has the corresponding property.

This is the same notion as \emph{log \'etale} and \emph{kummer \'etale} of \cite{berkovichcomplex}.
\end{definition}

The class of adically kummer \'etale morphisms will give us a good category and Grothendieck pre-topology to use to define the adically kummer \'etale site.

\begin{definition}
The \emph{small adically kummer \'etale site} $\site\mathfrak{X}_{ak\acute{e}t}$ of a log special formal scheme $\mathfrak{X}$ will be the category whose underlying objects are adically kummer \'etale morphisms $\mathfrak{Y} \rightarrow \mathfrak{X}$, whose morphisms are commutative triangles. This is given the pre-topology generated by jointly surjective families.
The associated topos will be denoted $\topos\mathfrak{X}_{ak\acute{e}t}$, and the associated hypercomplete $\infty$-topos will be denoted $\hootopos \mathfrak{X}_{ak\acute{e}t}$.
\end{definition}

We will start with some lemmas about the basic algebraic structures involved here.

\begin{lemma}\label{formal-etale-structure-lemma}
Let $f: \mathfrak{X} \rightarrow \mathfrak{Y}$ be a morphism of separated locally noetherian special formal  schemes. Assume that it can be realized as a completion $f = \hat{F}$ of a map $F : X \rightarrow Y$ of schemes locally of finite type and flat over $\spec R$, with subschemes $Z \subset Y_s$ and $Z' \subset X_s$, such that $Z'$ is open in $F^{-1}(Z)$.

Explicitly, we factor $f$ as
\[ \hat{X}_Z' \rightarrow \hat{X}_{F^{-1}Z} \rightarrow \hat{Y}_Z \]
Then $F$ is classically \'etale on a neighborhood of $Z'$ if and only if $f$ is adically \'etale.
\end{lemma}

\begin{proof}
The key idea here is that $\mathcal{O}_{\mathfrak{X},x}$ and $\mathcal{O}_{X,x}$ have the same completions as local rings. The algebraizability of the schemes implies that the algebra $\mathcal{O}_{\mathfrak{X},x} \to \mathcal{O}_{\mathfrak{Y},f(x)}$ is of finite type, and is not just topologically of finite type. We deduce that $\mathcal{O}_{X,x} \to \mathcal{O}_{Y,F(x)}$ is \'etale if and only if the map $\mathcal{O}_{\mathfrak{X},x} \to \mathcal{O}_{\mathfrak{Y},f(x)}$ is. We only must show that being adically \'etale and being of finite type implies formally \'etale for local rings. But we can check the \'etaleness of a finite type morphism of local rings on the special fiber, which being adically \'etale implies.
\end{proof}

As a corollary, an \'etale morphism of schemes $Y \rightarrow X$ induces an adically \'etale morphism $\mathfrak{Y}\rightarrow\mathfrak{X}$. If we enrich these schemes by giving them strict log structures over $\spec R$, a natural question is if a log \'etale or kummer \'etale morphism completes to an adically log \'etale or adically kummer \'etale morphism.

\begin{proposition}
The completion of a log \'etale morphism $f : X \rightarrow Y$ of coherent log schemes $X$ and $Y$ which are individually strict over $\spec R$ gives an adically log \'etale morphism $f : \mathfrak{X} \rightarrow \mathfrak{Y}$ where both special formal schemes are endowed with their canonical log structures.
\end{proposition}

\begin{proof}
By the above observation, we can immediately reduce to the case where $X$ and $Y$ have charts $P$ and $Q$ respectively, in which case we must merely note that $Y \times_{\spec \mathbb{Z}[Q]}\spec \mathbb{Z}[P]$ is canonically isomorphic to $Y \times_{\spec R[Q]} \spec R[P]$, the completion of which is  exactly $Y \widehat{\times}_{\spf R\{Q\}} \spf R\{P\}$. Thus if the map $X \rightarrow Y \times_{\spec \mathbb{Z}[Q]} \spec \mathbb{Z}[P]$ is classically \'etale, then the induced map $\mathfrak{X} \rightarrow Y \widehat{\times}_{\spf R[Q]} \spf R[P]$ is adically \'etale.
\end{proof}

The following lemma will be important in a later comparison theorem as it allows us to pass between the formalisms of log schemes and of log formal schemes.

\begin{lemma}\label{finitekummerlemma}
Let $(\mathfrak{X},M) = \spf A$ and $(\mathfrak{Y},N)=\spf B$ be the formal spectra of special $R$-algebras with fine vertical log structures. Assume that they are given global charts $\mathscr{P} \rightarrow A$ and $\mathscr{Q} \rightarrow B$. Further assume that $B$ is a finite $A$-algebra. Then a morphism of log formal schemes $f: \mathfrak{Y} \rightarrow \mathfrak{X}$ is adically kummer \'etale if and only if the natural map $F$ from $X = \spec A$ to $Y = \spec B$ is kummer \'etale as a map locally of finite type of fine noetherian log schemes.
\end{lemma}

\begin{proof}
In such a chart, being adically kummer is that the map of monoids $P \rightarrow Q$ itself is kummer and that
\[ \mathfrak{Y} \rightarrow \mathfrak{X} \widehat{\times}_{\spf R\{P\}} \spf R\{Q\} \]
is adically \'etale. However, $\spf R\{Q\}$ is a finite $\spf R\{P\}$ algebra, as we will see in the proof of the next proposition. This implies that the above completed tensor product is the usual tensor product. Thus the map $Y \rightarrow X$ factors through $X \times_{\spec R[P]} \spec R[Q]$ and the map to this fibered product is \'etale if and only if its completion is formally \'etale.
\end{proof}

Now we can come to the machinery of passing between adically kummer \'etale topos and the quasi-\'etale topos. This proposition is essentially the same as one in \cite{berkovichcomplex}

\begin{proposition}
The generic fiber of an adically kummer \'etale morphism $f : \mathfrak{X} \rightarrow \mathfrak{Y}$ between log special formal schemes with vertical log structures is quasi-\'etale.

This gives a morphism of sites from the adically kummer \'etale site of $\mathfrak{X}$ to the quasi-\'etale site of the associated generic fiber, $\site (\mathfrak{X}_{ak\acute{e}t}) \rightarrow \site(\mathfrak{X}_{\eta,q\acute{e}t})$, inducing a map of topoi $\topos(\mathfrak{X}_{\eta,q\acute{e}t}) \rightarrow \topos(\mathfrak{X}_{ak\acute{e}t})$ 
\end{proposition}

\begin{proof}
The generic fiber of an adically \'etale morphism is quasi-\'etale as a morphism of non-archimedean analytic spaces, and being quasi-\'etale is local for the quasi-\'etale topology. We may then reduce the problem adically \'etale locally on both the base and source. This allows us to assume that both $\mathfrak{X}$ and $\mathfrak{Y}$ are given charts $P$ and $Q$ respectively and that we can factor the map through the pull-back

\begin{center}
\begin{tikzcd}
\mathfrak{X} \arrow[d, "f'"] & \\
\mathfrak{Y}\hat{\times}_{\spf R\{Q\}} \spf R\{P\} \arrow[r] \arrow[d, "p"] & \spf R \{ Q \} \arrow[d] \\
\mathfrak{Y} \arrow[r] & \spf R\{ P \}
\end{tikzcd}
\end{center}

The map  $f'$ is adically \'etale, and so its generic fiber is quasi-\'etale. The map $p$ will have \'etale generic fiber, simply because the algebra map of affinoid algebras $K \otimes_R R\{Q\} \to K \otimes_R R\{P\}$ is \'etale. The composition is then still quasi-\'etale.
\end{proof}

The above induces a map of sites which induces a map on $\notp$-profinite completions of the shapes of the associated $\infty$-topoi,
\[ \profqettoptype_{\notp} \mathcal{X} \rightarrow \proffkettoptype_{\notp} \mathfrak{X} \]
which we will study.

\subsection{The comparison theorem}

Let $\mathfrak{X}$ be an algebraizably log smooth fs log formal scheme. We can zariski-locally cover $\mathfrak{X}$ by special formal affines, that is by open formal subschemes of the form $\spf A$ for some special adic $R$-algebra $A$. Furthermore, we can work adically \'etale locally to assume that the algebra $A$ is the completion of an affine log scheme $V= \spec S$ flat and finite type  $\spec R$ along some closed subset of the special fiber $Z \subset V_s$. We can further assume that $V/\spec R$ is a fine vertical log smooth morphism, and we can further assume that the log structure on both $V$ and $\mathfrak{X}$ is given by a chart $P \rightarrow \mathcal{O}_\mathfrak{X}=A$.

Write $X=\spec A$ for the affine noetherian log scheme given by the commutative ring $A$ along with its log structure induced by $P \rightarrow A = \mathcal{O}_X$.

We begin with the following theorem.

\begin{theorem}
The scheme $X$ is log regular.
\end{theorem}

\begin{proof}
By Proposition 7.1 of \cite{kato1994toric}, it is enough to check along the closed points of $X$. These are in bijection with the closed points of $\mathfrak{X}$. For a closed point $x \in X$, we can now study the local ring $\mathcal{O}_{X,x}$. This local ring admits an ideal $I_x$ generated by $M_x \setminus \mathcal{O}_{X,x}^*$. In the local case, we can also describe $I_x$ in terms of the chart, as the ideal $I_x = (P \setminus P \cap \mathcal{O}_{X,x}^*) \mathcal{O}_{X,x}$. For $X$ to be log regular at $x$, we want the local ring $\mathcal{O}_{X,x}/I_x$ to be regular and to have the following equality hold.
\[ \dimension \mathcal{O}_{X,x} = \dimension \mathcal{O}_{X,x}/I_x + \rank_\mathbb{Z} (M^{gp})_x \setminus\mathcal{O}_{X,x}^*  \]
The left-hand side is always bounded above by the right hand side. We claim that this equality holds if and only if the following equality does.
\[ \dimension \hat{\mathcal{O}}_{X,x} = \dimension \hat{\mathcal{O}}_{X,x}/J_x + \rank_\mathbb{Z} (N^{gp})_x\setminus\hat{\mathcal{O}}_{X,x}^* \]
Here $N$ is the log structure induced on the ring $\hat{\mathcal{O}}_{X,x}$ by the composition $P \rightarrow \mathcal{O}_{X,x} \rightarrow \hat{\mathcal{O}}_{X,x}$, and $J_x$ is the ideal generated by $N_x \setminus \hat{\mathcal{O}}_{X,x}^*$. As above, this admits a description as $J_x = (P \setminus P \cap \hat{\mathcal{O}}_{X,x}^*) \hat{\mathcal{O}}_{X,x}$, but because the map $\mathcal{O}_{X,x} \rightarrow \hat{\mathcal{O}}_{X,x}$ is injective on units, we see that the two sets of generators are equal
\[ P \setminus P \cap \hat{\mathcal{O}}_{X,x}^* =  P \setminus P \cap \mathcal{O}_{X,x}^* \]
Thus $I_x \hat{\mathcal{O}}_{X,x} = J_x$. This implies that the completion of $\mathcal{O}_{X,x}/I_x$ is isomorphic to $\hat{\mathcal{O}}_{X,x}/J_x$, and thus their dimensions are equal.

We are then reduced to worrying about the rank term. We see that the stalk of $M^{gp}$ at $x$ must be the groupification of $M_x$ as groupification preserves colimits. Since both $N$ and $M$ are log structures associated with the same pre-log structure, we see that the two abelian groups are isomorphic.

The second equality holds by applying the above argument to the spectrum of the ring $S$ along the point corresponding to $x$, noting that $S$ with the log structure induced by $P$ is log regular, as it is log smooth over a log regular base.
\end{proof}

Now we can state the local comparison theorem.

\begin{theorem}
\label{local-comparison-berkovich-log}
Let the assumptions on $\mathfrak{X}$ be the local ones introduced at the very beginning of this subsection. The $\notp$-profinitely completed quasi-\'etale homotopy type of the strictly $K$-affinoid space $\mathcal{X}=\mathfrak{X}_\eta$ is equivalent to the $\notp$-profinitely completed adically kummer \'etale homotopy type of its formal model $\mathfrak{X}$. More explicitly We have a chain of equivalences of profinite spaces

\[ \profettoptype_\notp \mathcal{X} \simeq \profqettoptype_\notp \mathcal{X} \simeq \profettoptype_\notp X_K \simeq \profettoptype^{k\acute{e}t}_\notp X^{k\acute{e}t} \simeq \proffkettoptype _\notp \mathfrak{X}_{ak\acute{e}t} \]
\end{theorem}

We will break the proof up into two pieces, one regarding the fundamental group and one regarding abelian cohomology.

\begin{proof}[Proof for fundamental pro-groupoids]
The first isomorphism of groupoids is an immediate consequence of Theorem \ref{comparison-berkovich-adic}. The second is Kiehl's theorem, and the third follows from a result of K. Fujiwara and K. Kato in an unpublished preprint. See Proposition B.7 of \cite{hoshi2009exactness} for a statement and proof. The fourth isomorphism follows by coherent descent of modules for affine noetherian formal schemes, plus the Lemma \ref{finitekummerlemma}.
\end{proof}

The cohomological case proven in more generality in \cite{berkovichcomplex}, and the method of proof given here is quite similar in spirit. The difference is that we pass through the formalism of adic spaces.

\begin{proof}[Proof for cohomology]
We must check that the cohomology of each topos with locally constant coefficients agrees with that of each other one. Let us explicitly write the chain of maps on the level of cohomology.
\[ H^i(\mathcal{X}_{\acute{e}t}) \overset{(1)}{\rightarrow} H^i(\mathcal{X}_{q\acute{e}t}) \overset{(2)}{\leftarrow} H^i(\mathcal{X}^{ad}_{\acute{e}t}) \overset{(3)}{\leftarrow} H^i(X_K^{\acute{e}t}) \overset{(4)}{\leftarrow} H^i(X_{k\acute{e}t}) \overset{(5)}{\leftarrow} H^i(\mathfrak{X}_{ak\acute{e}t}) \]
The reader should beware the adic space we added in the above list.

We first argue the case for constant coefficients. From left to right, the first equivalence follows from Corollary \ref{comparison-qet-et}. The second isomorphism is also a consequence of Theorem \ref{comparison-berkovich-adic}. The third isomorphism follows from Corollary 3.2.2 of \cite{huber2013etale} if the coefficients are constant. In general we base change along a finite \'etale cover on  which the system is constant, and use the Hochschild-Serre spectral sequence to deduce it for the original space $\mathfrak{X}$. This reduction was suggested in \cite{hansen-mo-answer}. The fourth isomorphism follows from the work of K. Fujiwara, K. Kato, and Ch. Nakayama. The fifth isomorphism is a bit more complicated, and we will argue for it in the next paragraph.

If $\mathfrak{X}$ is topologically of finite type, then it follows as the class of kummer \'etale algebras of the affine log scheme $X$ and the adically kummer \'etale algebras of the affine log formal scheme $\mathfrak{X}$ coincide by Th\'eor\`eme 7 of \cite{elkik1973solutions}. See also Theorem 3 of section 4 of \cite{hoshi2009exactness}. The last isomorphism then follows since the topoi are equivalent. If it is \emph{not} topologically of finite type, we must be slightly more careful. We will show that $\lambda = (3) \circ (4) \circ (5)$ is an isomorphism. Since both $(3)$ and $(4)$ are isomorphisms, $(5)$ must be as well. The group homomorphism $\lambda$ is induced by the generic fiber functor of topoi. If we are completing $\mathfrak{X}$ along some subscheme $Y \subset \tilde{\mathfrak{X}}$, to obtain $\mathfrak{Y} = \hat{\mathfrak{X}}_Y$ then we can apply the isomorphism of Proposition 3.15 of \cite{huberfiniteness} (cf. Theorem 3.1 of \cite{berkovich1996vanishing} as well) to conclude the cohomology of $\mathcal{X}^{ad}_{\acute{e}t}$ and that of $\mathfrak{X}_{ak\acute{e}t}$ are isomorphic.
\end{proof}

By applying Proposition \ref{notpprofiniteshapeequivalence}. Now we can globalize the result.

\begin{theorem}\label{globalcomparisonberklog}
Let $\mathfrak{X}$ be a locally noetherian separated special formal scheme with canonical log structure. Assume that $\mathfrak{X}$ is fine as a log formal scheme, and also algebraizably formally $R$-log smooth. Put $\mathcal{X}$ for its generic fiber considered as a $K$-analytic space. Then we have an equivalence of the $\notp$-profinite quasi-\'etale homotopy type of $\mathcal{X}$ and the $\notp$-profinite adically kummer-\'etale homotopy type of $\mathfrak{X}$.

\[ \profqettoptype_{\notp} \mathcal{X} \rightarrow \proffkettoptype_{\notp} \mathfrak{X} \]
\end{theorem}

\begin{proof}
Cover $\mathfrak{X}$ by an affine cover $\mathfrak{X} = \cup_i \mathfrak{U}_i$ where each $\mathfrak{U_i}$ satisfies the hypothesis of the local comparison theorem above. We can then take the associated simplicial formal scheme 
\[ \mathfrak{U}_\bullet = \cosk_0 \Big[ \coprod_i \mathfrak{U}_i \rightarrow \mathfrak{X} \Big] \]
which is an \'etale hypercover of $\mathfrak{X}$. Further, by the separated hypothesis, we see that the simplicial formal scheme $\mathfrak{U}_\bullet$ is levelwise affine, and that each level also meets the hypothesis of the local comparison theorem.

We can now consider $h^\mathfrak{U}_\bullet$ as a simplicial object of the hypercomplete adically \'etale $\infty$-topos $\hootopos \mathfrak{X}_{a\acute{e}t}$. Now, post-compose this with the `slice topos' functor
\[ \Delta^{op} \rightarrow \hootopos \mathfrak{X}_{a\acute{e}t} \rightarrow \rtop / \hootopos \mathfrak{X}_{a\acute{e}t} \]
This functor preserves colimits by Proposition 6.3.5.14 of \cite{lurie2009higher}. The shape functor is defined as a left adjoint and thus preserves colimits. Furthermore the $\notp$-profinite completion functor is also a left adjoint and thus also preserves colimits.

Now we see that
\[ \profqettoptype_{\notp} \mathcal{X} \simeq \colim_{[n] \in \Delta^{op}} \profqettoptype_{\notp} \mathcal{U}_{[n]} \]
\noindent and similarly that
\[ \proffkettoptype_{\notp} \mathfrak{X} \simeq \colim_{[n] \in \Delta^{op}} \proffkettoptype_{\notp} \mathfrak{U}_{[n]} \]
But we have a natural transformation between the two diagrams of $\notp$-profinite shapes that induces a level wise equivalence by the local comparison above. Thus it glues together into a global equivalence of $\notp$-profinite simplicial sets.
\end{proof}

Picking a homotopy inverse to $\qettoptype \mathcal{X} \rightarrow \ettoptype \mathcal{X}$ and $\notp$-profinitely completing it immediately implies the following corollary.

\begin{corollary}
Let $\mathfrak{X}$ be a locally noetherian separated special formal scheme with canonical log structure. Assume that $\mathfrak{X}$ is fine as a log formal scheme, and also algebraizably formally $R$-log smooth. Put $\mathcal{X}$ for its generic fiber considered as a $K$-analytic space.

There is an equivalence of $\notp$-profinite spaces

\[ \profettoptype_{\notp} \mathcal{X} \simeq \proffkettoptype_{\notp} \mathfrak{X} \]
\end{corollary}

It also will imply the following, although we will include a proof to explain the gap from Theorem \ref{globalcomparisonberklog} to it.

\begin{corollary}
Let $V_R$ be a locally of finite type log smooth $R$-scheme over a complete discrete valuation ring $R$, and let $p$ be the residue characteristic of $R$. Further assume that the morphism $V_R \to \spec R$ is a strict log morphism, so that in particular $V_K = V_R^{triv}$.

Then the open immersion of the two analytifications induces an equivalence of $\notp$-profinite \'etale homotopy types

\[ \profettoptype_{\notp} \big( \widehat{V_R} \big)_\eta \simeq \profettoptype_{\notp} \big(V_K\big)^{an} \]
\end{corollary}

\begin{proof}
By Theorem \ref{artin-comparison-ell}, we have
\[ \profettoptype_{\notp} \big(V_K\big)^{an} \simeq \profettoptype_{\notp} \big(V_K\big) \]
by the work of Fujiwara, Kato, and Nakayama we have that the $\notp$-profinite \'etale homotopy type of $V_K$ is equivalent to the $\notp$-profinite kummer \'etale homotopy type of $V_R$.
\[ \profettoptype_{\notp} \big(V_K\big) \simeq \profiniteshape_{\notp} \big(\hootopos V_R^{k\acute{e}t} \big) \]
Moving the other way in the diagram, Theorem \ref{globalcomparisonberklog} states that
\[ \profettoptype_{\notp} \big( \widehat{V_R} \big)_\eta \simeq \proffkettoptype_{\notp} \widehat{V_R}\]
and finally we observe that
\[ \proffkettoptype_{\notp} \widehat{V_R} \simeq \profiniteshape_{\notp} \big( \hootopos V_R^{k\acute{e}t}\big)\]
since the cohomology of both computes nearby cycles.
\end{proof}

\bibliographystyle{plain}
\bibliography{./references}

\end{document}